\documentclass[reqno]{amsart}

\usepackage{a4wide,eucal,enumerate,mathrsfs}

\usepackage{eucal,enumerate,mathrsfs}
\usepackage{amsmath,amssymb,epsfig,bbm,amsthm}
\usepackage{ifthen}

\usepackage{pdfsync}
\usepackage[usenames]{color}

\numberwithin{equation}{section}

\newtheorem*{mainconjecture}{Conjecture}
\newtheorem{theorem}{Theorem}[section]

\newtheorem{corollary}[theorem]{Corollary}
\newtheorem{lemma}[theorem]{Lemma}
\newtheorem{remark}[theorem]{Remark}
\newtheorem{proposition}[theorem]{Proposition}
\newtheorem{definition}[theorem]{Definition}
\newtheorem{notation}[theorem]{Notation}

\newcommand{\optsquared}[1]{\ifthenelse{\equal{#1}{}}{}{[#1]}}
\newcommand{\argmin}{\operatorname{argmin}}
\newcommand{\R}{\mathbb R}
\newcommand{\N}{\mathbb N}
\newcommand{\rmd}{\mathrm d}
\newcommand{\rme}{\mathrm e}
\newcommand{\eps}{\varepsilon}
\renewcommand{\H}{\mathbb H}
\newcommand{\range}[1]{\ifthenelse{\equal{#1}{}}{R}{\mathrm R[#1]}}
\newcommand{\dom}[1]{\ifthenelse{\equal{#1}{}}{D}{\mathrm D[#1]}}

\newcommand{\rangestar}[1]{\ifthenelse{\equal{#1}{}}{R_\star}{\mathrm R_\star[#1]}}

\newcommand{\image}[1]{\ifthenelse{\equal{#1}{}}{J}{\mathrm J[#1]}}
\newcommand{\imagestar}[1]{\ifthenelse{\equal{#1}{}}{J_\star}{\mathrm J_\star[#1]}}
\newcommand{\critical}[1]{\ifthenelse{\equal{#1}{}}{S}{\mathrm S[#1]}}

\newcommand{\sfz}{\mathsf z}
\newcommand{\sfx}{\mathsf x}
\newcommand{\sfy}{\mathsf y}
\newcommand{\sfu}{\mathsf u}

\newcommand{\sft}{\mathsf t}

\newcommand{\solution}[1]{\mathrm{GF}\optsquared{#1}}

\newcommand{\appsolution}[1]{\mathrm{AGF}\optsquared{#1}}
\newcommand{\truncatedsolution}[1]{\mathrm{TGF}\optsquared{#1}}

\newcommand{\minsolution}[1]{\mathrm{GF}_{min}\optsquared{#1}}

\newcommand{\cU}{\mathcal U}

\newcommand{\cK}{\mathcal K}

\newcommand{\rmC}{\mathrm C}
\newcommand{\Lip}{\mathrm{Lip}}

\newcommand{\MM}{\mathrm{MM}}
\newcommand{\GMM}{\mathrm{GMM}}
\newcommand{\sfd}{\mathsf d}
\newcommand{\dsfd}[1]{\sfd_{#1}}
\newcommand{\sfD}{\mathsf D}
\newcommand{\dsfD}[1]{\sfD_{#1}}

\newcommand{\LL}[1]{L}

\newcommand{\DMM}[3]{\mathrm{M}_{#1}(#2;#3)}
\newcommand{\MS}[3]{\mathrm{MS}_{#1}(#2;#3)}

\newcommand{\nn}[2]{N(#1,#2)}

\newcommand{\nchi}{{\raise.3ex\hbox{$\chi$}}}
\newcommand{\vphi}[3]{\varphi_{#1,#2}(#3)}
\newcommand{\vPhi}[3]{\Phi_{#1,#2}(#3)}
\newcommand{\vY}[4]{\mathcal J_{#1,#2,#3}(#4)}

\newcommand{\restr}[1]{\lower3pt\hbox{$|_{#1}$}}
\newcommand{\minimal}{minimal}
\newcommand{\Minimal}{Minimal}
\newcommand{\minimality}{minimality}
\newcommand{\fourteen}{14}

\usepackage{hyperref}

\title[Reverse approximation of gradient flows as Minimizing
Movements]{Reverse approximation of gradient flows as Minimizing
Movements: a conjecture by De Giorgi}

\begin{document}

\author{Florentine Fleissner} \thanks{Technische Universit\"at M\"unchen;
  email:
  \textsf{fleissne@ma.tum.de}}
  
\author{Giuseppe Savar\'e}
 \thanks{Dipartimento di Matematica ``F.~Casorati'', Universit\`a di Pavia; email:
 \textsf{giuseppe.savare@unipv.it}.
 G.S.~has been partially supported by
Cariplo foundation and Regione Lombardia
through the project 2016-1018 ``Variational evolution problems and optimal transport''}.


\begin{abstract} 
  We consider the Cauchy problem for the gradient flow
  \begin{equation}
    \label{eq:81}
    \tag{$\star$}
    u'(t)=-\nabla\phi(u(t)),\quad t\ge 0;\quad
    u(0)=u_0,
  \end{equation}
  generated by a continuously differentiable 
  function $\phi:\H\to\R$ in a Hilbert space $\H$ and study the
  reverse approximation of solutions to (\ref{eq:81}) by the De Giorgi
  Minimizing Movement approach. 

  We prove that if $\H$ has finite
  dimension
  and $\phi$ is quadratically bounded from below 
  (in particular if $\phi$ is Lipschitz)
  then for \emph{every} solution $u$ to (\ref{eq:81}) (which may have
  an infinite number of solutions) 
  there exist perturbations $\phi_\tau:\H\to\R \ (\tau>0)$
  converging to $\phi$ in the Lipschitz norm such that $u$ can be
  approximated by the Minimizing Movement scheme 
  generated by the recursive minimization of $\Phi(\tau,U,V):=\frac 1{2\tau}|V-U|^2+ \phi_\tau(V)$:
  \begin{equation}\label{eq:abstract}\tag{$\star\star$}
    U_\tau^n\in \argmin_{V\in \H} \Phi(\tau,U_\tau^{n-1},V)\quad n\in\N, \quad U_\tau^0:=u_0.
  \end{equation}
  We show that the piecewise constant interpolations with time step $\tau >
  0$ 
  of \emph{all} possible selections of solutions $(U_\tau^n)_{n\in\N}$ to
  (\ref{eq:abstract}) 
  will converge to $u$ as $\tau\downarrow 0$. 
  This result solves a question raised by Ennio De Giorgi in
  \cite{DeGiorgi93}. 
  
  We also show that even if $\H$ 
  has infinite dimension the above approximation holds for the distinguished class of \minimal\
  solutions to \eqref{eq:81},
  that generate all the other solutions to \eqref{eq:81} by time reparametrization.
\end{abstract}

\maketitle

{\small\tableofcontents}

\section{Introduction}

In his highly inspiring paper \cite{DeGiorgi93}
Ennio De Giorgi introduced the variational
notion
of \emph{Minimizing Movement} in order to present a 
general and unifying approach to a large class of evolution problems
in a vector, metric or even topological framework.

In the case of time-invariant evolutions in a topological space $\H$,
Minimizing Movements can be characterized by the recursive minimization of a functional
$\Phi:(0,\infty)\times \H\times \H\to [-\infty,+\infty]$.
For a given initial datum $u_0\in \H$ and a parameter
$\tau>0$
(which plays the role of discrete time step size) 
one looks for sequences
$(U^n_\tau)_{n\in \N}$ such that $U^0_\tau:=u_0$ and for every $n\ge
1$
\begin{equation}\normalcolor
  \label{eq:76}
  \Phi(\tau,U^{n-1}_\tau,U^n_\tau)=\min_{V\in
    \H}\Phi(\tau,U^{n-1}_\tau,V),\quad
  \text{i.e.}\quad
  U^n_\tau\in \operatorname{argmin}\Phi(\tau,U^{n-1}_\tau,\cdot).
\end{equation}
Any sequence satisfying \eqref{eq:76} {\normalcolor gives rise} to a 
\emph{discrete solution} $U_\tau:[0,\infty)\to\H$ at time step $\tau$, 
obtained by piecewise constant
interpolation of the values $(U^n_\tau)_{n\in \N}$:
\begin{equation}
 \label{eq:48}
 U_\tau(0):=U^0_\tau=u_0, \quad U_\tau(t):= U^n_\tau\quad\text{if }t\in ((n-1)\tau,n\tau],\quad
  n\in \N. 
\end{equation} 
A curve $u:[0,+\infty)\to\H$ is called Minimizing Movement
associated to $\Phi$ with initial datum $u_0$ (short $u\in\MM(\Phi, u_0)$) if   
there exist discrete solutions $U_\tau$ 
{\normalcolor(for
$\tau$ in a right neighborhood of $0$)} 
to the scheme (\ref{eq:76}) converging pointwise to $u$ 
as $\tau\downarrow 0$:
\begin{equation}
  \label{eq:68}
  u(t)=\lim_{\tau\downarrow0}U_\tau(t)\quad\text{for every }t\ge0.
\end{equation}
A curve $u:[0,+\infty)\to\H$ is more generally called Generalized Minimizing Movement (short $u\in\GMM(\Phi, u_0)$) 
if there exist a suitable vanishing subsequence $k\mapsto\tau(k)$ of
time steps and 
corresponding discrete 
solutions $U_{\tau(k)}$ at {\normalcolor time step} $\tau(k)$ to (\ref{eq:76}) such that
\begin{equation}
  \label{eq:68bis}
  u(t)=\lim_{k\to\infty}U_{\tau(k)}(t)\quad\text{for every }t\ge0.
\end{equation}   
 
The general notion of Minimizing Movement scheme
has proved to be extremely useful in a variety of
analytic, geometric and physical contexts; we refer to 
\cite{Ambrosio95, AGS08, fleissner2016gamma}, \cite{braides2012local} and \cite{mielke2011differential, mielke2016balanced} for a more detailed account of some
applications and developments and to
the pioneering paper \cite{AlmgrenTaylorWang93} by Almgren, Taylor, and
Wang.

Perhaps the simplest (though still interesting) situation arises if 
$\H$ is a Hilbert space and one tries to implement the scheme 
\eqref{eq:76} to solve the Cauchy problem for the gradient flow
\begin{equation}
  \label{eq:73}
  u'(t)=-\nabla\phi(u(t)),\quad t\ge0,
\end{equation}
with initial datum $u_0$ and continuously differentiable driving functional $\phi:\H\to \R$.
In this case a natural choice for the functional $\Phi$ is 
\begin{equation}
  \label{eq:77}
  \Phi(\tau,U,V):=\frac1{2\tau}|V-U|^2+\phi(V),
\end{equation}
for which the scheme (\ref{eq:76}) represents a sort of iterated
minimization of $\phi$ perturbed by $\frac1{2\tau}|\cdot-U|^2$. 
The last term penalizes 
the squared distance 
(induced by the norm $|\cdot|$ of $\H$) from the previous minimizer
$U$.
The Euler equation associated with the minimum problem
\eqref{eq:76} is then given by
\begin{equation}
  \label{eq:78}
  \frac{U^n_\tau-U^{n-1}_\tau}\tau+\nabla\phi(U^n_\tau)=0,
\end{equation}
so that the Minimizing Movement scheme can be considered as a
variational formulation of the implicit Euler method applied to
\eqref{eq:73}.
It is then natural to compare the class of 
solutions to \eqref{eq:73} and 
the classes of Minimizing Movements $\MM(\Phi, u_0)$ and Generalized Minimizing Movements $\GMM(\Phi, u_0)$ for $\Phi$ as in (\ref{eq:77}).  

If $\phi$ is a convex (or a quadratic perturbation of a convex)
function, it is possible to prove (see e.g.~\cite{Brezis70,Ambrosio95,AGS08})
that the Minimizing Movement scheme (\ref{eq:76}) is convergent to the unique
solution $u$ of \eqref{eq:73} with initial datum $u_0$, i.e. $\MM(\Phi,u_0)=\{u\}$. 
This
fundamental result can be extended to general convex and lower semicontinuous
functions
$\phi$, possibly taking the value $+\infty$ at some point of $\H$,
provided \eqref{eq:73} is suitably formulated as a subdifferential
inclusion.
Convexity assumptions can also be considerably relaxed
\cite{MarinoSacconTosques89,Rossi-Savare06}
as well as the Hilbertian character of the distance (see e.g.~\cite{MarinoSacconTosques89,AGS08,Mielke-Rossi-Savare13}).

\subsection*{Minimizing Movements and gradient flows governed by
  $\rmC^1$ functions.}
If $\H$ is a finite dimensional Euclidean space and $\phi$ is a continuously differentiable Lipschitz function, or more generally, a continuously differentiable function satisfying the lower quadratic bound
\begin{equation}\label{eq:coercive intro}
  \exists\,\tau_*>0,\ \phi_*\in\R:\quad
  \frac 1{2\tau_*}|x|^2+\phi(x)\ge -\phi_*\quad \text{for every } x\in\H,
\end{equation}
it is not difficult to see that the set $\GMM(\phi,u_0)$ is not empty and that
all its elements are solutions to \eqref{eq:73}. 

In general, there are more than one solution to (\ref{eq:73}) with initial datum $u_0$. 
A notable aspect is that the set $\MM(\phi,u_0)$ may be empty and/or 
$\GMM(\phi,u_0)$ merely a proper subset of the class of
solutions to \eqref{eq:73} with initial datum $u_0$. Such peculiarities pointed out by De Giorgi can be observed even in
one-dimensional examples of gradient flows driven by $\rmC^1$ Lipschitz functions \cite{DeGiorgi93}.  

It is then natural to look for possible perturbations of the scheme 
associated with \eqref{eq:77}, generating \emph{all} the
solutions to \eqref{eq:73}:  
this property would deepen our understanding of a gradient flow as a minimizing motion.  
This kind of question has also been treated
in the different context of rate-independent evolution processes \cite{Mielke-Rindler09}
from which we borrow the expression 
\emph{reverse approximation}.

A first contribution \cite{GGS94,Gobbino99} in the framework of the Minimizing Movement
approach to (\ref{eq:73}) 
deals with a uniform approximation
of $\Phi$, 
based on allowing approximate minimizers in each step of the scheme generated by \eqref{eq:77}.  

A much more restrictive class of approximation was proposed by De
Giorgi, 
who made the following conjecture
\cite[Conjecture 1.1]{DeGiorgi93}:
\begin{mainconjecture}[De Giorgi '93]
  Let us suppose that $\H$ is a finite dimensional Euclidean space
  and $\phi:\H\to \R$ is a continuously differentiable Lipschitz
  function. A map $u\in \rmC^1([0,\infty);\H)$ is a solution of
  \eqref{eq:73}
  if and only if 
  there exists a family $\phi_\tau:\H\to \R$, $\tau>0$, of Lipschitz
  perturbations of $\phi$ such that 
  \begin{equation}
    \label{eq:79}
    \lim_{\tau\downarrow0}\Lip[\phi_\tau-\phi]=0,
  \end{equation}
  and for the corresponding generating functional
  \begin{equation}
    \label{eq:80}
    \Phi(\tau,U,V):=\frac{1}{2\tau}|V-U|^2+\phi_\tau(V)
  \end{equation}
  one has 
  $u\in \GMM(\Phi,u(0))$.
\end{mainconjecture}
$\Lip[\cdot]$ in \eqref{eq:79} denotes the Lipschitz seminorm
\begin{equation}
  \label{eq:40}
  \Lip\big[\psi\big]:=\sup_{x,y\in \H,\ x\neq
    y}\frac{\psi(y)-\psi(x)}{|y-x|}\quad
  \text{whenever }\psi:\H\to \R.
\end{equation}
One of the main difficulties of proving this property concerns 
the behaviour of $u$ at critical points $w\in\H$ where $\nabla\phi(w)=0$
vanishes. Since $\nabla\phi$ is just a $\rmC^0$ map, it might happen that $u$ reaches a critical point after finite time, stays there for some amount of time and then leaves the point again.   
An even worse scenario might happen if the $0$ level set of $\nabla\phi$ is not
discrete. 
Even in the one dimensional case it is possible to construct
functions $\phi:\R\to \R$ with a Cantor-like $0$ level set $K\subset \R$ of $\phi'$
and corresponding solutions $u$ parametrized by a finite measure 
$\mu$ concentrated on $K$ and singular with respect to the Lebesgue
measure (see Appendix \ref{sec:example} for an explicit example).

A second difficulty arises from the lack of stability of the
evolution, due to non-uniqueness: since even small perturbations
may generate quite different solutions, one has to find suitable
perturbations of $\phi$ 
that keep these instability effects under control.
\medskip
\paragraph{\textbf{Aim and plan of the paper}}
In this paper we address the question raised by De Giorgi and we give
a positive answer to the above conjecture, in a stronger form
(\textbf{Theorem \ref{thm: De Giorgi conjecture}}):
we will show that it is possible to find Lipschitz perturbations $\phi_\tau$ of
$\phi$ in such a way that \eqref{eq:79} holds and 
\begin{equation}\label{eq:AGF introduction}
\MM(\Phi,u(0))=\{u\}=\GMM(\Phi,u(0)) 
\end{equation}
for the corresponding generating functional $\Phi$ defined by
\eqref{eq:80}.
An equivalent characterization of (\ref{eq:AGF introduction}) can be
given in terms of the \emph{discrete solutions} to the scheme:
\emph{all} the discrete solutions $U_\tau$ of \eqref{eq:76} will
converge to $u$ as $\tau\downarrow 0$.
Our result also covers the case of a
$\rmC^1$ function $\phi$ 
satisfying the lower quadratic bound (\ref{eq:coercive intro}).

Moreover, this reverse approximation can also be performed if
$\H$ has infinite dimension, for a particular class of solutions
(Theorem~\ref{thm:appmax}),
which is still sufficiently general to generate all the possible
solutions
by time reparametrization (Theorem~\ref{thm: max gf}).

\medskip In order to obtain an appropriate reverse approximation, we will introduce and
apply new techniques that seem of independent interest and give further 
information
on the approximation of the gradient flows \eqref{eq:73} in a finite and infinite
dimensional framework.

In Section 2 we will collect some preliminary material
and we will give a detailed account of notions of
\emph{approximability} of gradient flows (Section \ref{sec: subsec 2.4}),
in particular the notion of \emph{strong approximability} (which is equivalent to (\ref{eq:AGF introduction}) in the finite dimensional case) and the notion of \emph{strong approximability in every compact interval $[0, T]$} (which appears to be more fitting in the infinite dimensional setting lacking in compactness).    

\medskip
A first crucial concept in our analysis is a \emph{notion of partial order
between solutions to \eqref{eq:73}}. Such notion plays an important role in
any situation where non-uniqueness phenomena are present.
The basic idea is to study the family of all the solutions $u$ that share
the same range $\range u=u([0,\infty))$ in $\H$. 
On this class it is possible to introduce a {\normalcolor natural partial order} by saying that 
$u\succ v$ if there exists an increasing $1$-Lipschitz map $\sfz:[0,\infty)\to[0,\infty)$ such
that $u(t)=v(\sfz(t))$ for every $t>0.$ 

We will show in Theorem \ref{thm: max gf} that for a given range
$R=\range u$ there is always a distinguished
solution $v$ (called \emph{\minimal}), which induces all the other ones by such time
reparametrization. This solution has the remarkable property to cross
the critical set of the energy $\{w\in \H:\nabla \phi(w)=0\}$ in a 
Lebesgue negligible set of times (unless it becomes
eventually constant after some time $T_\star$, 
in that case it has the property in $[0,T_\star]$).

This analysis will be carried out for $\rmC^1$ solutions {\normalcolor to the Cauchy problem \eqref{eq:73} for a gradient flow} in an infinite dimensional Hilbert
space, but it can be {\normalcolor considerably generalized and adapted for abstract evolution problems \cite{fleissner-in-preparation} 
including general gradient flows in metric
spaces (under standard assumptions on the energy functional and on its metric slope as in \cite{AGS08}) and generalized semiflows (which have been introduced in \cite{ball2000continuity}).}

\medskip
In Section \ref{subsec: 4.1} we will study the general
problem to find Lipschitz perturbations of $\phi$ which confine
the discrete solutions of the Minimizing Movement scheme to a given
compact set $\cU$. We will find that a `penalization' with the distance
from $\cU$ is sufficient to obtain this property. 
The important thing here will be a precise quantitative estimate of appropriate `penalty' coefficients 
depending on the respective time step and on a sort of approximate invariance of $\cU$. 

In Section \ref{subsec: 4.2} we will obtain a first result on the reverse approximation of gradient flows. We will prove that every \minimal\ solution to (\ref{eq:73}) is approximable in the strong form (\ref{eq:AGF introduction}) by applying the estimates from Section \ref{subsec: 4.1} to suitably chosen compact subsets of its range. 

This result can be extended to infinite dimensional Hilbert spaces,
even if in the infinite dimensional case, 
the existence of solutions to gradient flows of $\rmC^1$ functionals is not guaranteed a priori (existence of a solution can be proved if $\nabla\phi$ is weakly continuous, see \cite[Theorem 7]{browder1964non}).  
However, if a solution exists, it always admits a \minimal\ 
reparametrization
and our result can be applied.

\medskip
The reparametrization technique and the technique of confining
discrete solutions to a given compact set provide a foundation for the
reverse approximation of the gradient flows. The last crucial step is
a reduction to the one dimensional case
and its careful analysis.
The detailed study of the one dimensional situation will be performed in Section
\ref{sec: one dimensional setting}. We will find a smoothing argument that allows to 
approximate any solution to (\ref{eq:73}) by a sequence of \minimal\ solutions for perturbed energies.
We can then base our proof of the reverse approximation (\ref{eq:AGF
  introduction}) for arbitrary solutions 
on the approximation by \minimal\ solutions (which are approximable in the form (\ref{eq:AGF introduction})) instead of working directly on the
discrete scheme. 

In Section \ref{sec: conclusion}
the one-dimensional result is `lifted' to arbitrary finite dimension
by a careful use of the Whitney extension Theorem (this is the only
point where we need a finite dimension): in this way,
we will obtain the reverse approximation result (\ref{eq:AGF introduction}) for arbitrary solutions to (\ref{eq:73}) in finite dimension.

\subsubsection*{Acknowledgments}
Part of this paper has been written while the authors were visiting
the
Erwin Schr\"odinger Institute for Mathematics and Physics
(Vienna), 
whose support is gratefully acknowledged. The first author gratefully acknowledges support from the Global-Challenges-for-Women-in-Math-Science program at Technische Universit\"at M\"unchen.

\subsection*{List of main notation}

\ 

\medskip
\halign{$#$\hfil\ &#\hfil
\cr
\H,\ \langle\cdot,\cdot\rangle,\ |\cdot|&Hilbert space, scalar product
and norm;
\cr
\dsfd T(u,v),\ \dsfd\infty(u,v)&distances between functions in $\H^{[0,T]}$ and
$\H^{[0,\infty)}$, \eqref{eq:42bis}--\eqref{eq:42};
\cr
\dsfD T(v,\cU),\ \dsfD\infty(v,\cU)&distances between a map $v$ and a
collection of maps $\cU$, \eqref{eq:45T}--\eqref{eq:45};
\cr
\solution \phi&solutions of the gradient flow equation \eqref{eq:1};
\cr
\truncatedsolution\phi&truncated solutions of 
equation \eqref{eq:1}, see before Theorem \ref{thm: max gf};
\cr
\minsolution\phi&class of \emph{minimal solutions} to \eqref{eq:1}, 
Definition \ref{def: max gf};
\cr
\critical\psi&subset of $x\in\H$ where $\nabla\psi(x)=0$;
\cr
T_\star(u)&minimal time after which $u$ is definitely constant,
\eqref{eq:92};
\cr
U_\tau&piecewise constant interpolant of a minimizing sequence
$(U^n_\tau)_{n\in \N}$;
\cr
\Lip[\psi]&Lipschitz constant of a real map $\psi:\H\to\R$,
\eqref{eq:40};
\cr
\succ&partial order in $\solution\phi$, Definition \ref{def: max gf};
\cr
\Phi(\tau,U,V)&functional characterizing the Minimizing Movement
scheme, \eqref{eq:35}--\eqref{eq:36};
\cr
\MM(\Phi,u_0)&Minimizing Movements, \eqref{eq:37};
\cr
\GMM(\Phi,u_0)&Generalized Minimizing Movements, \eqref{eq:38bis};
\cr
\MS\tau\psi{u_0},\ \MS\tau\psi{u_0,N}&Minimizing sequences,
\eqref{eq:33} and Remark \ref{rem:bounded-interval};
\cr
\DMM\tau\psi{u_0},\ {\normalcolor\DMM\tau\psi{u_0,T}}&Discrete solutions,
\eqref{eq:34} and Remark \ref{rem:bounded-interval};
\cr
\nn\tau T&$\min\{n\in \N:n\tau\ge T\}$, Remark
\ref{rem:bounded-interval};
\cr
\cU(\tau,T)&sampled values of a map $u$, \eqref{eq:55}
\cr
}
\section{Notation and preliminary results}
\label{sec:preliminary}

\subsection{Vector valued curves and compact convergence}
\label{subsec:compact-open}
Throughout the paper, let $(\H,\langle\cdot,\cdot\rangle)$ be a
Hilbert space with norm $|\cdot|:=\sqrt{\langle\cdot,\cdot\rangle}$.

A function $\psi:\H\to \R$ is Lipschitz 
if $\Lip[\psi]<\infty$, where $\Lip[\cdot]$ has been defined in \eqref{eq:40}.
$\Lip(\H)$  
will denote
the vector space of Lipschitz 
real functions 
on $\H$. 

$\rmC^1(\H)$ will denote the space of 
continuously differentiable real functions:
by Riesz duality, the differential $\mathrm
D\psi(x)\in \H'$
of $\psi\in \rmC^1(\H)$ at a point $x\in \H$ can be represented by a vector $\nabla
\psi(x)\in \H$. The set of stationary points will be denoted by
$\critical\psi:=\{v\in \H:\nabla
\phi(v)=0\}$.
Notice that 
a function in $\psi
\in \rmC^1(\H)$ belongs to $\Lip(\H)$ if and only if $x\mapsto
|\nabla \psi(x)|$ is bounded in $\H$.


Let $T\in (0,\infty)$; we introduce a distance on the vector space
$\H^{[0,T]}$ (resp.~$\H^{[0, +\infty)}$) 
of curves defined on $[0,T]$ (resp.~$[0, +\infty)$) with values in
$\H$ by setting
\begin{equation}
  \label{eq:42bis}
  \dsfd T(u,v):=\sup_{t\in [0,T]} 
  \Big(|u(t)-v(t)|\land
  1\Big)\quad
  \text{for every }u,v:[0,T]\to \H,
\end{equation}
\begin{equation}
  \label{eq:42}
  \dsfd\infty(u,v):=\sup_{t\ge 0} \frac 1{1+t}\Big(|u(t)-v(t)|\land
  1\Big)\quad
  \text{for every }u,v:[0,\infty)\to \H.
\end{equation}
$\dsfd T$ clearly induces the topology of uniform convergence on the interval
$[0,T]$.
It is not difficult to show that the distance $\dsfd\infty$ induces the
topology of compact convergence, i.e. the topology of uniform convergence on compact sets of $[0, +\infty)$: for every $T>0$ we have
\begin{equation}
  \label{eq:43}
  (1+T)^{-1}\dsfd T(u\restr{[0,T]},v\restr{[0,T]})
  \le 
  \dsfd\infty(u,v)
  \le \dsfd T(u\restr{[0,T]},v\restr{[0,T]})\lor (1+T)^{-1}
\end{equation}
so that a net $(u_\lambda)_{\lambda\in \Lambda}$ in $\H^{[0,\infty)}$
is
$\dsfd\infty$ convergent if and only if it is
convergent in the topology of compact convergence.

We will denote by $\range u$ the range of a function.

$\mathrm C^k([0,+\infty);\H)$ will be 
the vector space of $\mathrm
C^k$ curves with values in $\H$.
We will consider $\rmC^0([0,\infty);\H)$ as a (closed) subspace of 
$\H^{[0,\infty)}$ with the induced topology. We introduced the
distance \eqref{eq:42} on the bigger space $\H^{[0,\infty)}$ since we
will also consider
(discontinuous) piecewise constant paths with values in $\H$. 


For $T > 0$ and $\cU\subset\H^{[0, T]}, \ v\in\H^{[0, T]}$ we set  
\begin{equation}
  \label{eq:45T}
\dsfD T(v, \cU):= \sup_{u\in\cU}\dsfd T(v,u);\quad
\dsfD T(v,\cU):=+\infty \text{ if $\cU$ is empty,}
\end{equation}
and, similarly, 
for $\cU\subset\H^{[0, +\infty)}$ and $v\in\H^{[0, +\infty)}$ we define 
\begin{equation}
  \label{eq:45}
  \dsfD\infty(v,\cU):=\sup_{u\in \cU}\dsfd\infty(v,u);\quad
  \dsfD\infty(v,\cU):=+\infty \text{ if $\cU$ is empty.}
\end{equation}
Notice that $\dsfD T(v,\cU)$ (resp.~$\dsfD\infty(v,\cU)$)
is the Hausdorff distance
between the sets $\{v\}$ and $\cU$ induced by $\dsfd T$
(resp.~$\dsfd\infty$).

\subsection{Gradient flows}
\label{subsec:GF}
Let $\phi\in \rmC^1(\H)$ be given. 
$\solution\phi $ is defined as the collection of all 
curves $u\in \mathrm C^1([0, +\infty);\H)$ solving the gradient flow
equation
\begin{equation}
\label{eq:1}
u'(t) = -\nabla\phi(u(t)) 
\tag{GF}
\end{equation} 
in $[0,\infty)$. 
%
%
Let us collect some useful properties for $u\in\solution\phi$ which directly follow from the gradient flow equation (\ref{eq:1}). 

We first observe that $u$ satisfies
\begin{equation}
  \label{eq:3}
  |u'(t)|^2=|\nabla \phi(u(t))|^2=-\frac{\mathrm d}{\mathrm d
    t}\phi\circ u(t)\quad\text{for every }t\ge0,
\end{equation}
and thus  
\begin{equation}
  \label{eq:4}
  \phi(u(t_1))-\phi(u(t_2))=\int_{t_1}^{t_2}|u'(t)|^2\,\mathrm d t=
  \int_{t_1}^{t_2}|\nabla\phi(u(t))|\,|u'(t)|\,\mathrm d t=
  \int_{t_1}^{t_2}|\nabla\phi(u(t))|^2\,\mathrm d t
\end{equation}
for every $0\le t_1\le t_2$. 
In particular $\phi\circ u$ may take the same value at two points 
$t_1<t_2$ iff
$u$ takes a constant stationary value in $[t_1,t_2]$.

If $\phi$ is Lipschitz it is immediate to check that $u$ is also
Lipschitz and satisfies
\begin{equation}
  \label{eq:85}
  |u(t_2)-u(t_1)|\le \Lip[\phi]\,|t_2-t_1|\quad\text{for every
  }t_1,t_2\in [0,\infty).
\end{equation}
More generally, when $\phi$ satisfies \eqref{eq:coercive intro}, we
easily get
\begin{displaymath}
  \frac\rmd{\rmd t}\big(\phi(u(t))+\frac 1{\tau_*}|u(t)|^2+\phi_*\big)\le 
  \frac{1}{\tau_*^2} |u(t)|^2\le  \frac 2{\tau_*}\big(\phi(u(t))+\frac 1{\tau_*}|u(t)|^2+\phi_*\big)
\end{displaymath}
so that Gronwall Lemma and \eqref{eq:coercive intro} yield
\begin{equation}
  \label{eq:86}
  |u(t)|^2\le 2\tau_* C_0
  \mathrm e^{2t/\tau_*},\quad
  \phi(u(0))-\phi(u(t))\le C_0
  (1+\mathrm e^{2t/\tau_*}),
  \quad C_0
  :=\phi(u(0))+\frac1{\tau_*}|u(0)|^2+\phi_*.
\end{equation}
By applying H\"older inequality to \eqref{eq:4}
we thus obtain
\begin{equation}
  \label{eq:5}
  |u(t_2)-u(t_1)|
  \le \sqrt{C_0(1+\rme^{2T/\tau_*})} \,|t_2-t_1|^{1/2}\quad 
  \text{for every }t_1,t_2\in [0,T].
\end{equation}
\begin{lemma}
  \label{le:preliminary}
  Let $u\in \solution\phi$.
  \begin{enumerate}[(i)]
  \item $\range u$ is a connected set and the map
    $u:[0, +\infty)\to \range u$ is locally invertible around any
    point $x\in \range u\setminus \critical \phi$.
  \item
    $\range u$ is locally compact and it is
    compact if and only if $\phi$ attains its minimum in $\range u$ at
    some point $\bar u=u(\bar t)$ and $u$ is constant for
    $t\ge \bar t$.
  \item 
    The restriction of $\phi$ to $\range u$ is a homeomorphism with
    its image
    $\phi(\range u)\subset \R$.
  \item 
    If $u$ is not constant, then $\range u\setminus \critical\phi$ is dense in $\range u$, and
$\phi(\range u\setminus \critical\phi)$ is dense in $\phi(\range u)$.
  \end{enumerate}
\end{lemma}
\begin{proof}
  \emph{(i)} is obvious. 
  In order to show \emph{(ii)}, let us fix
    $\bar u=u(\bar t)\in \range u$; if $\phi$ attains its minimum in
    $\range u$ at $\bar u$ then $\phi(u(t))=\phi(u(\bar t))$ for every
    $t\ge \bar t$ and {\normalcolor \eqref{eq:4} yields} that
    $\bar u\in \critical \phi$ and $u(t)\equiv u(\bar t)$ for every
    $t\ge \bar t$, so that $\range u$ is compact. The converse
    implication is obvious.
    
    If $\phi|_{\range u}$ does not take its minimum at
    $\bar u=u(\bar t)\in \range u$, there exists some $t_1>\bar t$
    such that $\delta:=\phi(\bar u)-\phi(u(t_1))>0$. Since $\phi$ is
    continuous, the set
    $U:=\{u\in\H \ | \ \phi(u)\ge \phi(\bar u)-\delta\}$ is a closed
    neighborhood of $\bar u$ and $\range u \cap U = u([0,t_1])$ which
    is compact.

    \medskip\noindent
    \emph{(iii)} By \eqref{eq:4} the restriction of $\phi$ to $\range u$ is 
    continuous and
    injective. In order to prove that it is an homeomorphism
    {\normalcolor it is sufficient to prove that 
    $\phi|_{\range u}$
    is proper, i.e.~the counter image of every compact set in $J:=\phi(\range u)\subset \R$ is
    compact.}
    This property is obvious if $\range u$ is compact;
    if $\range u$ is not compact, then $\phi$ does not take its
    minimum on $\range u$ and $J$ is an interval of the form
    $(\varphi_-,\varphi_+]$ where $\varphi_+=\phi(u(0))$ 
    and $\varphi_-=\inf_{\range u}\phi$.
    Therefore any compact subset of $J$ is included in an interval 
    of the form
    $[\phi(u(\bar t)),\varphi_+]$ and its counter image 
    is a closed subset of the compact set $u([0,\bar t])$
    (recall that $u$ is constant in each interval where $\phi\circ u$
    is constant).
    
    \medskip\noindent
    \emph{(iv)} Let $\varphi_i=\phi(u(t_i))$, $i=1,2$, be two distinct
    points in $\phi(\range u\cap \critical\phi)$. 
    Assuming that $t_1<t_2$, \eqref{eq:4} shows that 
    there exists a point $\bar t\in (t_1,t_2)$ such that 
    $\nabla\phi(u(\bar t))\neq 0$, so that 
    $\bar\varphi=\phi(u(\bar t))$ belongs to 
    $(\varphi_2,\varphi_1)\setminus \phi(\critical\phi)$.
    We deduce that $\phi(\range u\setminus \critical\phi)$ is dense
    in $\phi(\range u)$ and, by the previous claim, 
    that 
    $\range u\setminus \critical\phi$ is dense in $\range u$.
\end{proof}
\begin{notation}
  If $u\in \rmC^0([0,\infty);\H)$ 
  we will set
  \begin{equation}
    \label{eq:92}
      T_\star(u):=
      \inf\big\{t\in [0,+\infty):
      u(s)=u(t)\quad\text{for every }s\ge t\big\}
  \end{equation}
  with the usual convention $T_\star(u):=+\infty$ if the 
  argument of the 
  infimum in \eqref{eq:92} is empty.
\end{notation}
It is not difficult to check that the map $\normalcolor \mathcal T_\star: \rmC^0([0,\infty);\H) \to [0, +\infty], \ u\mapsto T_\star(u),$ 
is lower semicontinuous with respect to 
the topology of compact convergence in $\rmC^0([0,\infty);\H)$.

By Lemma \ref{le:preliminary}(ii), if $u\in \solution \phi$ then
  \begin{equation}
    T_\star(u)<\infty\quad\text{if and only if}\quad
    \text{$\range u$ is compact};
\label{eq:84}
\end{equation}
if $\range u$ is compact
then $u(t)=u_\star:=u(T_\star(u))$ for every $t\ge
T_\star(u)$ and $u_\star$ is a stationary point.


\subsection{Minimizing movements}\label{subsec: 2.3}
Let 
a function $\psi:\H\to\R$, 
a time step $\tau>0$, and an initial value $u_0\in \H$ be given.

We consider
the (possibly empty) set 
$\MS\tau\psi{u_0}$ 
of \emph{Minimizing Sequences} $(U^n_\tau)_{n\in \N}$
such that
$U^0_\tau=u_0$ and 
\begin{equation}
  \label{eq:33}
  \frac1{2\tau}|U^n_\tau-U^{n-1}_\tau|^2+\psi(U^n_\tau)\le 
  \frac1{2\tau}|V-U^{n-1}_\tau|^2+\psi(V)
  \quad
  \text{for every }V\in \H, \quad n\ge 1.
\end{equation}
We can associate a discrete sequence satisfying \eqref{eq:33} with its piecewise constant interpolation $U_\tau:[0, \infty)\to
\H$ given by 
\begin{equation}
  \label{eq:34}
  U_\tau(0) := u_0,\qquad
  U_\tau(t):=U^n_\tau\quad\text{if }t\in ((n-1)\tau,n\tau].
\end{equation}
(\ref{eq:34}) can be equivalently expressed as 
\begin{displaymath}
U_\tau(t):=\sum_{n\in \N}U^n_\tau \nchi(t/\tau-(n-1)) \quad \text{for } t > 0,
\end{displaymath}
in which $\nchi:\R\to\R$ denotes the characteristic function of the interval
$(0,1]$. 
We call $\DMM\tau{\psi}{u_0}$ the class of discrete solutions 
$U_\tau$ at
time step $\tau>0$, which admit the previous representation \eqref{eq:34} in terms
of solutions to (\ref{eq:33}). 
%
\begin{remark}[Bounded intervals]
  \label{rem:bounded-interval}
  \upshape
  Sometimes it will also be useful to deal with approximations defined
  in a bounded interval $[0,T]$, involving finite minimizing
  sequences. 
  For $N\in \N$ we call 
  \begin{equation}
  \text{$\MS\tau\psi{u_0,N}$ 
  the set of sequences $(U^n_\tau)_{0\le n\le N}$
  satisfying \eqref{eq:33}.}
\label{eq:94}  
\end{equation}
  Similarly, 
  for a given a final time
  $T\in (0,+\infty)$ we set
  \begin{equation}
    \label{eq:95}
  \normalcolor  \nn\tau T:=\min \{n\in \N:n\tau\ge T\},
  \end{equation}
  and we define $\DMM\tau{\psi}{u_0, T}$ 
  as the collection of all the
  piecewise constant functions $U_\tau: [0, T]\to\H$ satisfying
  \eqref{eq:34} in their domain of definition, for some
  $(U_\tau^n)_n\in \MS\tau\psi{u_0,\nn\tau T}$.
\end{remark}

Let us now assign a
family of functions $\phi_\tau: \H \to \R$ depending on the
parameter $\tau\in (0,\tau_o)$ 
and
define the functional 
$\Phi:(0,\tau_o)\times \H\times \H\to \R$ as
\begin{equation}
  \label{eq:35}
  \Phi(\tau,U,V):=\frac1{2\tau}|V-U|^2+\phi_\tau(V). 
\end{equation}
\eqref{eq:33} for the choice $\psi:=\phi_\tau$ is equivalent to 
\begin{equation}
  \label{eq:36}
  U^0_\tau:=u_0;\qquad
  U^n_\tau\in \operatorname{argmin}\Phi(\tau,U^{n-1}_\tau,\cdot)\quad
  \text{for every }n\ge 1.
\end{equation}
According to \cite{DeGiorgi93}, a curve $u:[0,+\infty)\to\H$ is called
a Minimizing Movement associated to $\Phi$ if 
there exist discrete solutions $U_\tau\in 
\DMM\tau{\phi_\tau}{u_0}$ such that 
\begin{equation}
  \label{eq:37}
  \lim_{\tau\downarrow0} U_\tau(t)=u(t)\quad\text{for every }t\ge0.
\end{equation}
$\MM(\Phi,u_0)$ denotes the collection of all the Minimizing
Movements.

A curve $u: [0, +\infty)\to\H$ is called a Generalized Minimizing Movement \cite{DeGiorgi93} associated to $\Phi$  
if there exist a decreasing sequence $k\mapsto
\tau(k)\downarrow0$ 
and corresponding $U_{\tau(k)}\in
\DMM{\tau(k)}{\phi_{\tau(k)}}{u_0}$ such that 
\begin{equation}
  \label{eq:38bis}
  u(t)=\lim_{k\to\infty} U_{\tau(k)}(t) \quad\text{for every }t\ge0.
\end{equation}
$\GMM(\Phi,u_0)$ denotes the collection of all the Generalized Minimizing
Movements. It clearly holds that $\MM(\Phi, u_0)\subset\GMM(\Phi, u_0)$. 

\begin{remark}[Quadratic lower bounds]
  \label{rem:QLB}
  \upshape
  If for some $\tau_*>0$
  \begin{equation}
    \label{eq:112}
    \inf_V \Phi(\tau_*,u_0,V)=-A>-\infty
  \end{equation}
  (this happens, in particular, if 
  $\MS\tau{\phi_{\tau_*}}{u_0,N}$ is nonempty)
  and $ \Lip[\phi_{\tau_*}-\phi]\le \ell$,
  then $\phi$ 
  satisfies the quadratic lower bound 
  \begin{equation}
    \label{eq:106}
    \phi(x)\ge (\phi(u_0)-
    \phi_{\tau_*}(u_0))
    -A-\ell/2 -\frac{\ell\tau_*+1}{2\tau_*}|x-u_0|^2 \quad \text{for every }
    x\in\H.
  \end{equation}
  This shows that the lower quadratic bound 
  of \eqref{eq:coercive intro} is a natural assumption in the
  framework of 
  minimizing movements.
  \eqref{eq:106} follows by 
  the fact that 
  $\phi(x)-
  \phi(u_0)-(\phi_{\tau_*}(x)-\phi_{\tau_*}(u_0))\ge
  -\Lip[\phi_\tau-\phi_{\tau_*}]\,|x-u_0|\ge -\ell |x-u_0|$.
   
   Similarly, if $\phi$ satisfies \eqref{eq:coercive intro} and 
   $\Lip[\phi_\tau-\phi]\le \ell$, we get
   \begin{equation}
     \label{eq:113}
     \phi_\tau(x)\ge (\phi_\tau(0)-
     \phi(0))
     -\phi_*-\frac\ell2 -\frac{\ell\tau_*+1}{2\tau_*}|x|^2 \quad \text{for every }
    x\in\H.
   \end{equation}
\end{remark}
It is a well known fact that if $\phi\in \rmC^1(\H)\cap \Lip(\H)$ and
$\lim_{\tau\downarrow0}\Lip[\phi_\tau-\phi]=0$, then every 
$u\in
\GMM(\Phi,u_0)$ 
solves \eqref{eq:1} with initial
datum $u_0$. We present here the proof of this statement 
(including the case of $\phi\in\rmC^1(\H)$ satisfying 
\eqref{eq:coercive intro}) and a few related
results that will turn to be useful in the following.
\begin{lemma}[A priori estimates for minimizing sequences]
  \label{le:apriori}
  Let $\phi\in \rmC^1(\H)$ satisfy
  \eqref{eq:coercive intro}, let $\phi_\tau:\H\to\R$ be
  such that $\normalcolor \ell_\tau:=\Lip[\phi_\tau-\phi]<\infty$,
  let $T>0$ and $(U^n_\tau)_{0\le n\le N}
  \in \MS\tau{\phi_\tau}{u_0,N}$, $\normalcolor 1\le N\le\nn\tau T$.
  \begin{enumerate}[(i)]
  \item For every $1\le n\le N$ we have
    \begin{equation}
      \label{eq:97}
      \left|\frac{U^n_\tau-U^{n-1}_\tau}\tau+\nabla\phi(U^n_\tau)\right|\le 
      \ell_\tau.
    \end{equation}
  \item 
    Suppose that $\phi(u_0)\lor |u_0|^2\le  F$, $\ell_\tau\le 1$,
    $\tau\le \tau_*/16$.
    There exists a positive constant $C=C(\phi_*,\tau_*,F,T)$ 
    only
    depending
    on $\phi_*,\tau_*,F,T$, such that
    \begin{equation}
      \label{eq:111}
      \sup_{0\le n\le N}|U^n_\tau|^2\le C,\quad
    \normalcolor  \frac {1}{2\tau}\sum_{n=1}^N|U^n_\tau-U^{n-1}_\tau|^2
      \le \phi_\tau(u_0)-\phi_\tau(U^N_\tau)\le C
    \end{equation}
  \end{enumerate}
\end{lemma}
\begin{proof}
   \emph{(i)} Let us set $\psi_\tau:=\phi_\tau-\phi$.
   The minimality condition \eqref{eq:36} yields for every $W\in \H$
  \begin{displaymath}
    \phi(W)+\frac 1{2\tau}|U^{n-1}_\tau-W|^2-\phi(U^n_\tau)-\frac1{2\tau}|U^{n-1}_\tau-U^n_\tau|^2\ge
    \psi_\tau(U^n_\tau)-\psi_\tau(W)
    \ge -\ell_\tau |U^n_\tau-W|.
  \end{displaymath}
  We can choose $W:=U^n_\tau+\theta v$, divide the above inequality by $\theta>0$  and pass to the limit as $\theta\downarrow 0$ obtaining
  \begin{displaymath}
    \langle \tau^{-1}(U^n_\tau-U^{n-1}_\tau)+\nabla\phi(U^n_\tau),v\rangle\ge -\ell_\tau |v|\quad \text{for every $v\in \H$},
  \end{displaymath}
  which yields \eqref{eq:97}.
 
\medskip\noindent
\emph{(ii)} It follows by \cite[Lemma 3.2.2]{AGS08}, by 
using \eqref{eq:113} and $u_*:=0$. Up to the addition of 
a constant to $\phi_\tau$, it is not restrictive
to
assume that $\phi_\tau(0)=\phi(0)$.
\end{proof}
\begin{lemma}
  \label{le:well-known}
  Let $\phi\in \rmC^1(\H)$ 
  and let $\phi_\tau:\H\to\R$, $\tau\in (0, \tau_o)$, 
  such that $\lim_{\tau\downarrow0}\Lip[\phi_\tau-\phi]=0$. 
  \begin{enumerate}[(i)]
  \item If there exist a vanishing decreasing sequence 
    $k\mapsto \tau(k)$ and discrete solutions in a bounded interval
    $U_{\tau(k)}\in \DMM{\tau(k)}{\phi_{\tau(k)}}{u_0,T}$ such that 
    $u(t)=\lim_{k\to\infty}U_{\tau(k)}(t)$ for every $t\in [0,T]$, 
    then $u\in \rmC^1([0,T];\H)$ is a solution to \eqref{eq:1} with
    initial datum $u(0)=u_0$. 
  \item If $u\in \GMM(\Phi,u_0)$ then $u\in \solution\phi$.
  \item Let $T>0$ and $U_\tau\in \DMM\tau{\phi_\tau}{u_0,T}$, $\tau\in
    (0,\tau_o)$, be a family of discrete solutions taking values in a compact subset $\cK\subset \H$.
    Then for every decreasing and
    vanishing sequence
    $k\mapsto \tau(k)$ there exist a further subsequence (still denoted by
    $\tau(k)$) and a limit function $u\in \rmC^1([0,T];\H)$ such that 
    \begin{equation}
      \label{eq:96}
      \lim_{k\to\infty} \sfd_T(U_{\tau(k)},u)=0,\quad
      u\text{ is a solution to \eqref{eq:1} in $[0,T]$}.
    \end{equation}
    \item Let $U_\tau\in \DMM\tau{\phi_\tau}{u_0}$, $\tau\in
    (0,\tau_o)$, be a family of discrete solutions satisfying the following property: for every $T>0$ 
    there exist $\bar \tau\in (0,\tau_o)$ and a compact set
    $\cK\subset \H$ such that 
    $U_\tau([0,T])\subset \cK$ for every $\tau\in (0,\bar \tau)$.
    Then for every decreasing and
    vanishing sequence
    $k\mapsto \tau(k)$ there exist a further subsequence (still denoted by
    $\tau(k)$) and a limit function $u\in \rmC^1([0,\infty);\H)$ such that 
    \begin{equation}
      \label{eq:96bis}
      \lim_{k\to\infty} \sfd_\infty(U_{\tau(k)},u)=0,\quad
      u\in \solution\phi.
    \end{equation}
  \end{enumerate}
\end{lemma}
\begin{proof}
  \emph{(i)} 
  Let us call $\hat U_\tau(t)$ the piecewise linear interpolant
  of the values $U^n_\tau$, $0\le n\le \nn\tau T$ of the minimizing
  sequence associated to $U_\tau$:  
  $\hat U_\tau(t)=\sum_{n=0}^{\nn\tau T}U^n_\tau \hat\nchi(t/\tau-n)$ 
  where $\hat \nchi(t)=(1-|t|)\lor 0$.
  
  Since $\MS{\tau(k)}{\phi_{\tau(k)}}{u_0,T}$ are not empty and
  $\Lip[\phi_\tau-\phi]\to0$, by Remark \ref{rem:QLB}
  we deduce that $\phi$ satisfies the lower quadratic bound 
  \eqref{eq:coercive intro}. 
  By Lemma \ref{le:apriori}(ii) 
  we deduce that there exists $\tau_\star\in (0,\tau_o)$ sufficiently small
  such that any curve $\hat U_{\tau}$ is equi H\"older continuous for
  $\tau\le \tau_\star$,
  i.e.~there exists a constant $C$ independent of $\tau$ such that 
  \begin{equation}
    \label{eq:104}
    |\hat U_{\tau}(t)-\hat U_{\tau}(s)|\le C|t-s|^{1/2}\quad
    \text{for every }s,t\in [0,T],\ \tau\in (0,\tau_\star),
  \end{equation}
  and
  \begin{equation}
    \label{eq:105}
    \dsfd T(\hat U_\tau,U_\tau)\le C\sqrt \tau\quad
    \tau\in (0,\tau_\star).
  \end{equation}
  \eqref{eq:105} shows that $\hat U_\tau$ has the same limit points of
  $U_\tau$; since $\hat U_\tau$ is equi-H\"older, 
  the pointwise convergence $U_{\tau(k)}\to u(t)$ as $k\to\infty$
  for every $t\in [0,T]$ implies the uniform convergence of
  $U_{\tau(k)}$ and of $\hat U_{\tau(k)}$ to the same limit $u$, which 
  belongs to $\rmC^0([0,T];\H)$.

  Since $\hat U_\tau'(t)=\tau^{-1}(U^n_\tau-U^{n-1}_\tau)$ in each
  interval $((n-1)\tau,n\tau)$, we obtain from \eqref{eq:97}
  \begin{equation}
    \label{eq:100}
    \big|\hat U_\tau'(t)+\nabla\phi(U_\tau(t))\big|\le
    \Lip[\phi_\tau-\phi]
    \quad\text{for every }
    t\in [0,T]\setminus \{h\tau:0\le h\le \nn\tau T\}.
  \end{equation}
  We can then pass to the limit in (the integrated version of)
  \eqref{eq:100} for $\tau=\tau(k)$ 
  to obtain that 
  \begin{equation}
    \label{eq:101}
    u(t)=u_0-\int_0^t \nabla\phi(u(r))\,\rmd r\quad
    \text{for every }t\in [0,T],
  \end{equation}
  which shows that $u\in \rmC^1([0,T];\H)$ is a solution to \eqref{eq:1}.
  
  Let us remark that when $\phi$ is Lipschitz a reinforced version
  of \eqref{eq:104} and
  \eqref{eq:105}
  follows directly from \eqref{eq:97}, which yields
  \begin{equation}
    \label{eq:98}
    |U^n_\tau-U^{n-1}_\tau|\le L_\tau\tau\quad\text{for }1\le
    n\le \nn\tau T,\quad
    \dsfd T(U_\tau,\hat U_\tau)\le L_\tau\tau,\quad
    \normalcolor \Lip(\hat U_\tau)\le L_\tau,
  \end{equation}
  where $L_\tau:=\Lip[\phi]+\ell_\tau$.
  
  \medskip\noindent
  \emph{(ii)} The proof is completely analogous to (i).
  
  \medskip\noindent
  \emph{(iii)}
  We observe that $\hat U_{\tau(k)}$ takes values in the closed convex
  hull
  $\overline{\mathrm{co}(\cK)}$, which is still a compact subset of $\H$
  (see, e.g., \cite[Theorem 3.20]{Rudin91}). Since $k\mapsto
  U_{\tau(k)}$ is eventually equi-H\"older by \eqref{eq:104}, 
  Ascoli-Arzel\`a Theorem yields the relative compactness of the
  sequence
  in the uniform topology. We can apply the previous Claim (i).

  \medskip\noindent  
  \emph{(iv)}
  We can apply the previous point (iii) and a standard diagonal argument.  
\end{proof}
\subsection{De Giorgi conjecture and notions of approximability}\label{sec: subsec 2.4}
If the generating function
$\Phi$ of \eqref{eq:35} is
induced by perturbations $\phi_\tau$ converging to $\phi\in \rmC^1(\H)$ in
the Lipschitz seminorm, then 
Lemma \ref{le:well-known}(ii) shows that 
$\GMM(\Phi,u_0)\subset \solution\phi$.

The most challenging part of 
De Giorgi's conjecture deals with the opposite direction. 
It can be equivalently formulated in the following way:
\begin{quote}
  \begin{itshape} Suppose that $\H$ has finite dimension,
    $\phi\in \rmC^1(\H)\cap \Lip(\H)$, and let a solution
    $u\in \solution\phi$ be given. There exist a family of functions
    $\phi_\tau\in \Lip(\H)$, a decreasing sequence
    $k\mapsto \tau(k)\downarrow0$ and corresponding
    $U_{\tau(k)}\in \DMM{\tau(k)}{\phi_{\tau(k)}}{u(0)}$ such that
    \begin{equation}
      \label{eq:46}
      \lim_{\tau\downarrow0}\Lip[\phi_\tau-\phi]= 0,\quad
      \lim_{k\to\infty}\dsfd\infty(u,U_{\tau(k)})= 0.
    \end{equation}
  \end{itshape}
\end{quote}
We will introduce a stronger property, based on the set distance
introduced in \eqref{eq:45}.
\begin{definition}[Strongly approximable solutions]
  \label{def:AGF}
  Let $\phi\in \rmC^1(\H)$. 
  We say that a solution $u\in \solution\phi$ is a strongly approximable solution if 
  there exists a family of 
  perturbations
  $\phi_\tau:\H\to\R$, 
  $\tau\in (0,\tau_o)$, such that 
  \begin{equation}
    \label{eq:49}
    \lim_{\tau\downarrow0}\Lip[\phi_\tau-\phi]=0,\quad
    \lim_{\tau\downarrow0}\dsfD\infty(u,\DMM\tau{\phi_\tau}{u(0)})=0. 
  \end{equation} 
  We denote by $\appsolution\phi$ the class of strongly approximable solutions.
\end{definition}
The second part of \eqref{eq:49} is equivalent to the following property: for every $\tau>0$ sufficiently small the set 
$\DMM\tau{\phi_{\tau}}{u(0)}$ is nonempty and \emph{all}
the possible selections $U_\tau\in \DMM\tau{\phi_{\tau}}{u(0)}$
will converge to $u$ in the topology of compact convergence as $\tau\downarrow0$. We note that (\ref{eq:49}) implies  
\begin{equation}
  \label{eq:49 equiv}
    \MM(\Phi,u(0))=\{u\}=\GMM(\Phi, u(0))
\end{equation}
for the generating functional $\Phi$ of \eqref{eq:35}.
In the finite dimensional case, \eqref{eq:49 equiv}
is indeed equivalent to the second part of (\ref{eq:49}), due to the {\normalcolor $\sfd_\infty$-compactness} of every sequence $(U_\tau)$ of discrete solutions.

It is clear that any $u\in \appsolution\phi$ satisfies the property
expressed by De Giorgi's conjecture, and we will prove that in the finite dimensional Euclidean setting, indeed \emph{every} solution $u\in \solution\phi$ is strongly approximable.
%

In a few situations ($\H$ has infinite dimension and $\phi$ is not
bounded from below) we will also consider approximations on bounded
intervals,
recalling the notation introduced in Remark \ref{rem:bounded-interval}.
\begin{definition}\label{def:PAGF}
Let $\phi\in \rmC^1(\H)$. We say that a solution $u\in \solution\phi$ is strongly approximable in every compact interval if 
  there exists a family of Lipschitz
  perturbations
  $\phi_\tau\in \Lip(\H)$, $\tau\in(0,\tau_o)$, such that 
  \begin{equation}
    \label{eq:PAGF}
    \lim_{\tau\downarrow0}\Lip[\phi_\tau-\phi]=0,\quad
    \lim_{\tau\downarrow0}\dsfD T(u|_{[0, T]},\DMM\tau{\phi_\tau}{u(0),T})=0 \quad \text{for every } T>0. 
  \end{equation} 
\end{definition}
The notion of \emph{strong approximability in every compact interval} slightly differs from the notion of \emph{strong
  approximability} since we do not require the existence of elements
in 
$\DMM\tau{\phi_\tau}{u(0)}$ for $\tau > 0$ small enough and we work
with $\DMM\tau{\phi_\tau}{u(0), T}$ instead. 
The next remark better clarifies 
the relation between the two notions.
\begin{remark}[Strong approximability]
  \upshape
  \label{rem:approximability}
If a solution $u\in\solution\phi$ is strongly approximable 
in every compact interval
%
and \emph{for every sufficiently small $\tau > 0$
the set of minimizing sequences
$\MS\tau{\phi_\tau}{u(0)}$ is nonempty}, then $u$ is strongly
approximable according to Definition \ref{def:AGF}: it is a simple
consequence of \eqref{eq:43} and of the fact that 
for every $U\in \DMM\tau{\phi_\tau}{u(0)}$ 
the restriction $U\restr{[0,T]}$ belongs to 
$\DMM\tau{\phi_\tau}{u(0),T}$.

Conversely, if $u$ is strongly approximable 
and for every $\tau>0$ sufficiently small and
\emph{for every $N > 0$ 
any minimizing sequence in $\MS\tau{\phi_\tau}{u(0), N}$ 
can be extended to a 
minimizing sequence in $\MS\tau{\phi_\tau}{u(0)}$}, 
then $u$ is strongly approximable in every compact interval $[0,T]$,
according to Definition \ref{def:PAGF}. 

\emph{In the finite dimensional Euclidean case
the two notions of approximability are equivalent,}
since the minimization problems (\ref{eq:36}) are always solvable
for $\tau$ sufficiently small and 
$\phi$ quadratically bounded from below.
\end{remark} 


At the end of this preliminary section, we want to show that the class of strongly approximable solutions is closed with respect
to Lipschitz convergence of the functionals and compact convergence of the solutions. We make use of an equivalent characterization of
$\appsolution\phi$ provided by the next lemma. 
\begin{lemma}
  Let $\phi\in \rmC^1(\H)$.
  $u\in \appsolution\phi$ if and only if 
  for every $\eps>0$ there exist $\bar\tau>0$ 
  and a family $\phi_{\eps,\tau}:\H\to \R$, $0<\tau\le \bar
    \tau,$ such that
  \begin{equation}
    \label{eq:47}
    \Lip[\phi_{\eps,\tau}-\phi]\le \eps,\quad
    \dsfD\infty(u,\DMM\tau{\phi_{\eps,\tau}}{u(0)}\le \eps
    \quad\text{for every }\tau\in (0,\bar\tau].
  \end{equation}
\end{lemma}
\begin{proof}
  Since it is obvious that any $u\in \appsolution\phi$ satisfies the
  condition stated in the lemma, we only consider the inverse implication.
  
  Let us fix a decreasing sequence $\eps_n\downarrow0$; we can find a corresponding sequence $\bar
  \tau_n$
  and functions $\phi_{\eps_n,\tau}$ satisfying \eqref{eq:47} for
  $0<\tau\le \bar\tau_n$. By possibly replacing $\bar\tau_n$ with
  $\tilde\tau_n:=2^{-n}\land \min_{1\le m\le n}\bar \tau_m$, it is not restrictive to assume that $\bar
  \tau_n$ is also decreasing and converging to $0$. We can thus define
  \begin{equation}
    \label{eq:50}
    \phi_\tau:=\phi_{\eps_n,\tau}\quad\text{whenever }\tau\in
    (\bar\tau_{n+1},\bar \tau_n],
  \end{equation}
  and it is easy to check that this choice satisfies \eqref{eq:49}. \normalcolor The fact that $u\in\solution\phi$ follows by \eqref{eq:49 equiv} and Lemma \ref{le:well-known}(ii); hence, $u$ is a strongly approximable solution according to Definition \ref{def:AGF}. 
\end{proof}
\begin{lemma}
  \label{le:appsol}
  The class of strongly approximable solutions satisfies the following closure
  property:
  if $\phi,\phi_k\in \rmC^1(\H)$ 
  and $u_k\in
  \appsolution{\phi_k}$, $k\in \N$, with the same initial datum $\bar u=u_k(0)$,
  satisfy 
  \begin{equation}
    \label{eq:41}
    \lim_{k\to\infty} \Lip[\phi_k-\phi]=0,\quad
    \lim_{k\to\infty}\dsfd\infty(u_k,u)=0 
  \end{equation}
  then $u\in \appsolution\phi$.
\end{lemma}
\begin{proof}
  Let us fix $\eps>0$; according to \eqref{eq:41}, we can find $k\in
  \N$ such that 
  \begin{equation}
    \label{eq:51}
    \Lip[\phi_k-\phi]\le \eps/2,\quad
    \dsfd\infty(u,u_k)\le \eps/2.
  \end{equation}
  Since $u_k\in \appsolution{\phi_k}$ we can also find 
  $\bar \tau>0$ and a family of functions
  $\phi_{k,\eps,\tau}:\H\to\R$, $\tau\in (0,\bar \tau),$ such that 
  \begin{equation}
    \label{eq:52}
    \Lip[\phi_{k,\eps,\tau}-\phi_k]\le \eps/2,\quad
    \dsfD\infty(u_k,\DMM\tau{\phi_{k,\eps,\tau}}{u(0)})\le \eps/2\quad
    \text{for every }\tau\in (0,\bar \tau).
  \end{equation}
 \normalcolor The family $\phi_{k,\eps,\tau}$ obeys \eqref{eq:47}, since
  the triangle inequality yields
  \begin{gather*}
    \Lip[\phi_{k,\eps,\tau} - \phi]\le 
    \Lip[\phi_{k,\eps,\tau}-\phi_k]+\Lip[\phi_k-\phi]\le \eps,\\
    \dsfD\infty(u,\DMM\tau{\phi_{k,\eps,\tau}}{u(0)})\le 
    \dsfd\infty(u,u_k)+  \dsfD\infty(u_k,\DMM\tau{\phi_{k,\eps,\tau}}{u(0)})\le \eps.
  \end{gather*} 
\end{proof}

\section{The \minimal\ gradient flow} \label{sec: max gf}
In this section, 
we define and study a particular class of solutions to (\ref{eq:1})
for a function $\phi\in \rmC^1(\H)$,
which we call \emph{\minimal\ gradient flows}.  
Let us first introduce a partial order in $\solution\phi $.
\begin{definition}\label{def: max gf}
  If $u,v\in \solution\phi$ we say that 
  $u\succ v$ if $\range v\subset  \overline{\range u}
 $ 
 and there exists an increasing $1$-Lipschitz map $\sfz
 :[0,+\infty)\to [0,+\infty)$ with $\sfz(0)=0$ such that
 \begin{equation}
   \label{eq:6}
  0\le \sfz(t)-\sfz(s)\le t-s \quad \text{for every }0\le s\le t,
  \qquad
  u(t)=v(\sfz(t))\quad\text{for every }0\le t.
\end{equation}
An element $u\in \solution\phi $ is \minimal\ if for every $v\in \solution\phi $,  $u\succ v$ yields $u=v$. We will denote by $\minsolution\phi $ the
collection of all the \minimal\ solutions.
\end{definition} 

As it appears from \eqref{eq:6},
by `increasing' we mean that $\sfz(s)\le\sfz(t)$ for all $s\le t$; if we want to require a strict inequality, we will use the term `strictly increasing'. The same goes for `decreasing' and `strictly decreasing'.
\begin{remark}[Range inclusion]
  \label{rem:range}
\upshape
Notice that if $u\succ v$ then $\range u\subset \range v\subset
\overline{\range u}$; the inclusion $\range u\subset\range v$ is
guaranteed by \eqref{eq:6}. 

The condition $\range v\subset \overline{\range u}$ prevents some
arbitrariness in the extension of a candidate minimal solution. 
In order to understand its role, consider the classical $1$-dimensional example given by 
{\normalcolor $\phi'(x)=2\sqrt{|x|}$}. For a given $T_*>0$ the curve 
$  u(t):=\big((T_*-t)\lor 0\big)^2$ belongs to $\solution\phi$ and it is minimal 
according to the previous definition (it is an easy consequence
of the next Theorem \ref{thm: max gf}(5)).
However, the curve
$v(t):=u(t)-\big((t-2T_*)\lor 0\big)^2$  still belongs to
$\solution\phi$ and satisfies \eqref{eq:6} by choosing $\sfz(t)=t\land
T_*$.
$u\not\succ v$ since $\range v\not\subset \overline{\range u}=
[0,T^2_*]$.
\end{remark}

We note that constant solutions are minimal by definition. 

\begin{remark}[$\succ$ is a partial order in $\solution\phi$]
  \upshape
  It is easy to check that the relation $\succ$ is reflexive and
  transitive; let us show that it is also antisymmetric.
  If $u,v\in \solution\phi$ 
  satisfy $u\succ v$ and $v\succ u$, we can find increasing
  and $1$-Lipschitz maps 
  $\sfz_1,\sfz_2:[0,\infty)\to[0,\infty)$ such that 
  $u(t)=v(\sfz_1(t))$ and $v(t)=u(\sfz_2(t))$ for every $t\in
[0,\infty)$; in particular $u=u\circ \sfz$ where
$\sfz=\sfz_2\circ\sfz_1$ is also an increasing and $1$-Lipschitz
map satisfying $\sfz(t)\le t$.
Notice that the inequalities $\sfz_i(t)\le t$ and the monotonicity of
$\sfz_i$ yield
\begin{displaymath}
  \sfz(t)\le \sfz_i(t)\le t\quad
  \text{for every }t\ge0,\quad i=1,2.
\end{displaymath}
Let us fix $t\in [0,\infty)$: if $\sfz(t)=t$ then $\sfz_i(t)=t$
so that $u(t)=v(\sfz_1(t))=v(t)$.
If $\sfz(t)<t$ then 
$u$ is constant in the interval $[\sfz(t),t]$
so that $u(t)=u(\sfz_2(t))=v(t)$ as well.
\end{remark}
%

The next result collects a list of useful properties concerning
\minimal\ solutions. Recall that $T_\star(u)$ has been 
    defined by \eqref{eq:92}.
We introduce  the class of truncated solutions
$\truncatedsolution\phi\supset \solution\phi$ whose elements are 
solutions in $\solution\phi$ or curves 
$v:[0,\infty)\to\H$
of the form
$v(t):=\tilde v(t\land T)$ for some 
{\normalcolor $\tilde v\in \solution\phi$ and $T\in [0,\infty)$. The set $\solution\phi$ 
is closed in
$\mathrm C^0([0,\infty);\H)$.} 

\begin{remark}\label{rem:TGF}\upshape
If $v: [0, S] \to \H$ solves \eqref{eq:1} in $[0, S]$ for some $S > 0$ and there exists $u\in\solution\phi$ and $T>0$ so that $u(T) = v(S)$, then $v$ can be identified with an element in $\truncatedsolution\phi$ since $\tilde{v}\in\solution\phi$ where $\tilde{v}(t) := 
v(t) \text{ if } t\in[0, S]$ and  
$\tilde{v}(t) := u(t-S+T) \text{ if } t\in (S, +\infty)$. 
  
\end{remark}

\begin{theorem}\label{thm: max gf} {\normalcolor Let $\phi\in\rmC^1(\H)$ satisfy (\ref{eq:coercive intro}).} \ 
  \begin{enumerate}[\rm (1)] 
  \item For every $R=\range y\subset \H$ which is the range 
    of $y\in \solution\phi$ there exists a unique $u\in
    \minsolution\phi$ such that
    $R\subset\range u\subset \overline R$.
    If $v\in \solution\phi $ and 
    $R \subset \range v \subset \overline R$, then $v\succ u$.
    In particular, for every $v\in \solution \phi$ there exists a
    unique
    $u\in \minsolution\phi$ such that $v\succ u$.
  \item $u\in \minsolution\phi $ if and only if 
    for every 
    $v\in \truncatedsolution \phi$ with $v(0)=u(0)$ and
    $\range v\subset\overline{\range u}$ the following holds: if
    $u(t_0)=v(t_1)$ for some $t_0,t_1\ge0$ 
    then
    $t_0\land T_\star(u)\le t_1$.    
  \item $u\in \minsolution\phi $ if and only if
    for every 
    $v\in \truncatedsolution \phi$
    with $v(0) = u(0)$ and
    $\range v\subset \overline{\range u}$ we have
    $\phi(v(t))\ge \phi(u(t))$ for every $t\ge 0$.
  \item If $u\in \minsolution\phi$, 
    $v\in \truncatedsolution \phi$ with $v(0) = u(0)$,
    $\range v\subset \range u$ 
    and $\phi(v(t))\le \phi(u(t))$ for
    every $t\in [0,T_\star(v))$, 
    then $v(t)=u(t)$ for every $t\in [0,T_\star(v))$.
  \item 
    $u$ belongs to $\minsolution\phi $ 
    if and only if
    the restriction of 
    $u$ to $[0,T_\star(u))$
    crosses the set $\critical \phi$ of critical points of $\phi$ 
    in an $\mathscr{L}^1$-negligible set of times, i.e.
    \begin{equation}
      \label{eq:7}
      \mathscr L^1\Big(\{t\in[0, T_\star(u)): \ \nabla\phi(u(t)) = 0\}\Big)=0.
    \end{equation}
    If $u\in\minsolution\phi$ and $T_\star(u) > 0$, the map $t\mapsto (\phi\circ u)(t)$ is strictly
    decreasing in $[0, T_\star(u))$
    and the map $t\mapsto u(t)$ is injective in $[0,T_\star(u))$.
  \item A non-constant solution $u\in \solution\phi$ is minimal if and only if 
  there exists a
    locally absolutely continuous map 
    $$
    \text{$\psi:
    \big(\inf_{\range u}\phi,\phi(u(0))\big] \to \big[0,T_\star(u)\big)$\quad
    such that \quad
    $t=\psi(\phi(u(t))$ for every $t\in \big[0,T_\star(u)\big).$}$$
  \end{enumerate}
\end{theorem}

\begin{proof}
(1). 
\newcommand{\ww}{v}
\newcommand{\vv}{w}
Let us fix $\ww,y\in \solution\phi$, $R=\range y$ 
with $R \subset \range \ww \subset \overline R$; it is not restrictive to assume that $T_\star:=T_\star(\ww)>0$ (otherwise $\ww$
is constant and $R$ is reduced to one stationary point).
We set $\varphi_\star:=\inf_R\phi$ and we select a 
sequence $r_n\in R\setminus \critical \phi$ so that
$\varphi_n=\phi(r_n)$ is decreasing and converging to $\varphi_\star$.
We can find a corresponding increasing sequence of points
$T_n\to T_\star$ 
such that $\ww(T_n) = r_n$ and we set $R_n:=\big\{r\in R:\phi(r)\ge \varphi_n\}=\ww([0,T_n])$. We consider the class 
\begin{displaymath}
  \mathcal G[R_n]:=\Big\{\vv\in \truncatedsolution\phi: \range \vv= R_n\Big\}.
\end{displaymath}
$\mathcal G[R_n]$ is not empty since it
contains the function $t\mapsto \ww(t\land T_n)$, and the sublevel sets 
$\normalcolor \{\vv\in \mathcal G[R_n]:\mathcal T_\star(\vv)\le c\}$, $c > 0$, 
are compact in $\rmC^0([0,\infty);\H)$
by the Arzel\`a-Ascoli Theorem.
It follows that $\normalcolor \mathcal T_\star$ admits a minimizer in $\mathcal G[R_n]$ that we will
denote by $u_n$, with $S_n:=T_\star(u_n)\le T_n$. We now define
\begin{equation}
  \label{eq:10}
  \sfz_n(t):=\min\Big\{s\in [0,S_n]:u_n(s)=\ww(t)\Big\},\quad
  t\in [0,T_n],
\end{equation}
and we claim that $\sfz_n$ is increasing,
surjective
and $1$-Lipschitz from $[0,T_n]$ to $[0,S_n]$.
Since $\phi$ is an homeomorphism between $\ww([0,T_n])$ and the interval
$[\phi(\ww(T_n)),\phi(\ww(0))]$, $\sfz_n$ can be equivalently defined as 
$\min\Big\{s\in [0,S_n]:\phi(u_n(s))=\phi(\ww(t))\Big\}$, 
which shows that $\sfz_n$ is increasing. In order to prove that
$\sfz_n$ is $1$-Lipschitz, we argue by contradiction and suppose that there exist times
$0\le t_1<t_2\le T_n$ with $\delta_\sfz=\sfz_n(t_2)-\sfz_n(t_1)>
t_2-t_1=\delta_t$. Since by construction
$\ww(t)=u_n(\sfz_n(t))$, we can consider a new curve 
\begin{displaymath}
  \vv(r):=
  \begin{cases}
    u_n(r)&\text{if }0\le r\le \sfz_n(t_1),\\
    \ww(r+t_1-\sfz_n(t_1))&\text{if }\sfz_n(t_1)\le r\le
    \delta_t+\sfz_n(t_1)\\
    u_n(r+\delta_\sfz-\delta_t)&\text{if }r\ge \delta_t+\sfz_n(t_1).
  \end{cases}
\end{displaymath}
Defining $W_n:= S_n-\delta_\sfz+\delta_t<S_n$ it is easy to check
that $\vv(r)\equiv u_n(S_n)=\ww(T_n)$ for every $r\ge W_n$ and $\vv$ is a
solution to \eqref{eq:1} in the interval $[0,W_n)$, so that $\vv\in
\mathcal G[R_n]$ and $T_\star(\vv)=W_n<S_n$ which contradicts the
minimality of $u_n$.

The same argument shows that $u_n$ is in fact the unique minimizer of
$\normalcolor \mathcal T_\star$ in $\mathcal G[R_n]$: another minimizer $\tilde u_n$ will
also belong to $\mathcal G[R_n]$ with $T_\star (\tilde u_n)=S_n$,
so that there exists an increasing
$1$-Lipschitz map
$\mathsf r:[0,S_n]\to [0,S_n] $ such that $u_n(s)=\tilde u_n(\mathsf
r(s))$ for every $s\in [0,S_n]$. Since $\mathsf r (S_n)=S_n$ $\mathsf
r$ should be the identity so that $\tilde u_n$ coincides with $u_n$.

Let us now show that 
\begin{equation}
  \label{eq:9}
  S_n<S_{n+1},\quad
  u_n(s)=u_{n+1}(s),\quad \sfz_n(t)=\sfz_{n+1}(t)\quad 
  \text{for every }s\in [0,S_n],\ 
  t\in [0,T_n].
\end{equation}
In fact, for every $\bar t\in [0,T_n]$ with $\sfz_n(\bar t)=\bar s\in [0,S_n]$ 
there exists $\sfz_{n+1}(\bar t)=s'\in (0,S_{n+1})$ such that
$u_{n+1}(s')=\ww(\bar t)=u_n(\bar s)$; if $s' < \bar s$ we would conclude that the map
\begin{displaymath}
  \hat{\vv}(s):=
  \begin{cases}
    u_{n+1}(s)&\text{if }s\in [0,s'],\\
    u_n(s-s'+\bar s)&\text{if }s\ge s'
  \end{cases}
\end{displaymath}
belongs to $\mathcal G[R_n]$ with $T_\star(\hat{\vv})=S_n + s'-\bar s < S_n$ 
contradicting the minimality of $u_n$. Choosing $\bar s=S_n$
this in particular shows that $S_{n+1}>S_n$.
If $s'>\bar s$ we could define
\begin{displaymath}
  \tilde \vv(s):=
  \begin{cases}
    u_n(s)&\text{if }0\le s\le \bar s,\\
    u_{n+1}(s-\bar s+s')&\text{if }s\ge \bar s
  \end{cases}
\end{displaymath}
obtaining a function $\tilde \vv\in \mathcal G[R_{n+1}]$ with
$T_\star(\tilde \vv)=S_{n+1}-(s'-\bar s)<S_{n+1}$, contradicting the
minimality of $u_{n+1}$. We thus get $s' = \bar s$, and therefore
$u_{n+1}(s)=u_n(s)$ in $[0,S_n]$ and
$\sfz_n(t)=\sfz_{n+1}(t)$ in $[0,T_n]$.

Let us now set $S_\star:=\sup S_n$. Due to
\eqref{eq:9} we can define the maps 
\begin{equation}
  \label{eq:11}
  u(s):=
  \begin{cases}
    u_n(s)&\text{if }s\in [0,S_n]\text{ for some $n\in \N$,}\\
    u_\star&\text{if }s\in [S_\star,\infty)
  \end{cases}
  \quad
  \sfz(t):=
  \begin{cases}
    \sfz_n(t)&\text{if }t\in [0,T_n]\text{ for some $n\in \N$,}\\
    S_\star&\text{if }t\in [T_\star,\infty)
  \end{cases}
\end{equation}
with $u_\star := \lim_{s\uparrow S_\star}u(s)$ if $S_\star < \infty$. The curve $u$ solves \eqref{eq:1} in $[0,S_\star)$; 
{\normalcolor due to \eqref{eq:5}, the limit $u_\star$ is well-defined for $S_\star < \infty$}. If $T_\star<\infty$ then $ u_\star=\ww(T_\star)$; if
$T_\star=+\infty$ then $u_\star=\lim_{t\uparrow \infty}\ww(t)$ is a critical point of $\phi$ {\normalcolor as (\ref{eq:4}) yields $\int_0^\infty{|\nabla(\phi(v(t)))|^2 \ \mathrm d t} < +\infty$ in this case.} 
So we constructed an element $u\in \solution\phi$ with $\ww\succ u$ and
$R\subset \range u\subset \overline R$.

Notice that by construction $u$ just depends on $R$. 
Suppose now that there exists $\bar u\in \solution\phi$ with $u\succ \bar u$: in particular
$R\subset \range {\bar{u}}\subset \bar R$ by Remark \ref{rem:range}, 
so that the above argument shows that $\bar u\succ u$ and
therefore $\bar u\equiv u$. This property shows that $u\in
\minsolution\phi$.

\medskip\noindent
(2). Let us first observe that if $u\in \minsolution\phi$ is non-constant, then
the map $t\mapsto u(t)$ is injective in $[0, T_\star(u))$. In fact, if $u(t_0)=u(t_0+\delta)$ for
some $0\le t_0<t_0+\delta<T_\star(u)$, then $\phi\circ u$ and thus $u$ is constant in
$[t_0,t_0+\delta]$ so that the curve 
\begin{displaymath}
  w(t):=
  \begin{cases}
    u(t)&\text{if }t\in [0,t_0],\\
    u(t+\delta)&\text{if }t\ge t_0
  \end{cases}
\end{displaymath}
belongs to $\solution\phi$, satisfies $\range w=\range u$ and
$u(t)=w(\sfz(t))$ where $\sfz(t)=t\land t_0+(t-(t_0+\delta))_+$.
This yields $u\succ w$ so that $u\equiv w$ by
 the \minimality\ of $u$; we deduce $u(t)=u(t+\delta)$ 
for every $t\ge t_0$, which implies that $T_\star(u)\le t_0$,
a contradiction.

Let us now suppose that $u\in \minsolution\phi$,
$v\in \truncatedsolution \phi$ with $v(0) = u(0)$ and $\range v\subset
\overline{\range u}$. It is not restrictive to assume $T_\star(u)>0$.
We fix $t_0\ge0$ and $t_1\in [0,T_\star(v)]$ 
such that 
$u(t_0)=v(t_1)$,
and we define the curve
\begin{displaymath}
  w(t):=
  \begin{cases}
    v(t)&\text{if }0\le t\le t_1,\\
    u(t-t_1+t_0)&\text{if }t\ge t_1.
  \end{cases}
\end{displaymath}
We clearly have $w\in \solution\phi$;
moreover Lemma \ref{le:preliminary}(iii) yields 
$u([0,t_0])=v([0,t_1])$ so that $\range w=\range u$.
By the previous point (i) we deduce that 
$w\succ u$ and there exists an increasing $1$-Lipschitz map
$\sfz:[0,\infty)\to[0,\infty)$ 
such that $w(t)=u(\sfz(t))$. In particular,
$w(t_1)=v(t_1)=u(\sfz(t_1))=u(t_0)$ and $\sfz(t_1)\le t_1$. 
On the other hand, since $t\mapsto u(t)$ is injective
in $[0,T_\star(u))$ we deduce that 
$t_0=\sfz(t_1)$ or $\sfz(t_1)\ge T_\star(u)$.
If $t_1> T_\star(v)$ we simply replace $t_1$ by $T_\star(v)$,
since $v(t_1)=v(T_\star(v))$.

The converse implication is a simple consequence of the previous
claim: if $u\in \solution \phi$ we can construct the unique \minimal\ 
flow $v\in \minsolution \phi$ with $\range u\subset \range v\subset \overline{\range u}$, so that 
$u(t)=v(\sfz(t))$ for a suitable $1$-Lipschitz map satisfying
$\sfz(0)=0$. By assumption, $t\land T_\star(u)\le \sfz(t)$ but the $1$-Lipschitz
property yields $t\ge \sfz(t)$ so that $\sfz$ is the identity 
on $[0,T_\star(u))$. If $T_\star(u)=+\infty$ we deduce
immediately that $u\equiv v$; if $T_\star(u)<\infty$
we deduce that $v(t)=u(t)$ for every $t\in [0,T_\star(u)]$ 
and then $v\equiv u$ since $\range v\subset 
\range u=u([0,T_\star(u)])$.
In particular $u\in \minsolution\phi$.

\medskip\noindent
(3) Let $u\in \minsolution\phi$, 
$v\in \truncatedsolution \phi$
with $v(0) = u(0)$ and 
$\range v\subset \overline{\range u}$; it is not restrictive 
to assume $T_\star(u)>0$.
For every $t\in[ 0,T_\star(v))$ 
there exists $s\in [0,T_\star(u))$ such
that $v(t)=u(s)$. Claim (2) yields $s\le t$ so that 
$\phi(v(t))=\phi(u(s))\ge \phi(u(t))$. If $T_\star(v)<\infty$ we get
by continuity $\phi(v(T_\star(v)))\ge \phi(u(T_\star(v)))$ and
therefore
$\phi(v(t))=\phi(v(T_\star(v)))\ge \phi(u(T_\star(v)))\ge \phi(u(t))$
for every $t\ge T_\star(v)$.

In order to prove the converse implication, we argue as in the
previous claim and we construct the \minimal\ solution $v\in
\minsolution\phi$
with $\range u\subset \range v\subset \overline{\range u}$, so that 
$u(t)=v(\sfz(t))$ for a suitable $1$-Lipschitz map satisfying
$\sfz(0)=0$. Since $\sfz(t)\le t$ we get
$\phi(u(t))=\phi(v(\sfz(t)))\ge \phi(v(t))$, so that we deduce 
$\phi(u(t))=\phi(v(t))$ for every $t\ge0$; since $\phi$ is injective
on $\range v\supset \range u$ we obtain $u(t)=v(t)$.

\medskip\noindent
(4) {\normalcolor is an immediate consequence of the previous point (3)
and 
Lemma \ref{le:preliminary}(iii).}

\medskip\noindent
(5)
We first prove that a solution $u\in \solution\phi$ satisfying
\eqref{eq:7} is \minimal. In fact, if $u\succ v$ 
we can find a $1$-Lipschitz increasing map $\sfz$ such that 
$u(t)=v(\sfz(t))$. Since 
the map $\sfz$ is differentiable a.e.~in $[0, \infty)$ and 
$u,v$ are solutions to \eqref{eq:1} we obtain for a.e. $t\in [0,T_\star(u))$ 
\begin{equation*}
  -|\nabla\phi(u(t))|^2 \ = \ (\phi\circ u)'(t) \ = 
  \ (\phi\circ v\circ\sfz)'(t) \ = \ -|\nabla\phi(v(\sfz(t)))|^2
  \sfz'(t)\ = \
  -|\nabla\phi(u(t))|^2   \sfz'(t).
\end{equation*}
By \eqref{eq:7} we deduce $\sfz'(t) = 1$ {\normalcolor for a.e.~$t\in[0, T_\star(u))$, so that
$\sfz(t)=t$ in $[0, T_\star(u))$} and $v\equiv u$. 

Let us now prove that every $u\in \minsolution\phi$ satisfies
\eqref{eq:7}. Let $T_\star:= T_\star(u) > 0$.
Starting from $u$ we construct a solution $w\in \solution\phi$ 
with the same range as $u$ and which crosses $\critical\phi$ 
in an $\mathcal{L}^1$-negligible set of times. For this purpose, we
introduce the map
\begin{equation*}
\mathsf x\in \mathrm C^1([0, +\infty)), \quad \mathsf x(t) :=
\int_0^t{|u'(s)| \ \mathrm d s}\quad\text{with}\quad
X:=\int_0^\infty{|u'(s)| \ \mathrm d s}=\lim_{t\uparrow+\infty}\mathsf x(t),
\end{equation*}
and we consider the dense open set $\Omega:=\{t\in (0,T_\star):\mathsf x'(t)=|u'(t)|>0\}$.
Notice that $\mathsf x$ is strictly increasing in $[0,T_\star)$, since 
$\mathsf x(t_0)=\mathsf x(t_1)$ for some $0\le t_0<t_1<T_\star$ yields 
$u$ constant in $(t_0,t_1)$ which is not allowed by the \minimality\ of
$u$. 
We can thus define the continuous and strictly increasing inverse map $\mathsf y: [0, X) \rightarrow
[0,T_\star)$ such that $\mathsf y(\mathsf x(t))=t$ for every $t\in [0, T_\star)$.
We notice that the set
\begin{equation}
  \label{eq:17}
\Xi := \big\{x\in[0, X): \ u(\mathsf y(x)))\in \critical\phi\big\}
=  \mathsf x\big(\{t\in [0, T_\star): \ \mathsf x'(t) = 0\}\big)
\end{equation}
has Lebesgue measure $0$ by the Morse-Sard Theorem 
and that the map $\mathsf y$ is differentiable on its complement $[0, X)\setminus\Xi$ with
\begin{equation*}
  \mathsf y'(x) = \frac{1}{|u'(\mathsf y(x))|}=\frac{1}{|\nabla\phi(u(\mathsf y(x)))|}. 
\end{equation*}
Since $\mathsf y$ is continuous and increasing, its derivative belongs
to $L^1(0,X')$ for every $X'<X$. We can thus consider the strictly
increasing and locally absolutely continuous function
\begin{equation*}
\vartheta: [0, X) \rightarrow [0,\Theta), \quad \vartheta(x) :=
\int_{[0,x]\setminus \Xi}
{\frac{1}{|\nabla\phi(u(\mathsf y(r)))|} \ \mathrm d r},\quad
\Theta:=\int_{[0, X)\setminus \Xi}
{\frac{1}{|\nabla\phi(u(\mathsf y(r)))|} \ \mathrm d r}.
\end{equation*}
It holds that $\vartheta'(x) = \mathsf y'(x) > 0$ for every $x\in(0,
X)\setminus\Xi$
and $0<\vartheta(x_1)-\vartheta(x_0)\le  \mathsf y(x_1)-\mathsf y(x_0)$
for every $0\le x_0<x_1<X$, so that the composition $\sfz:=\vartheta\circ \mathsf x$ satisfies
\begin{equation}
  \label{eq:14}
  0<\sfz(t_1)-\sfz(t_0)\le t_1-t_0\quad\text{for every $0\le
    t_0<t_1<T_\star$.}
\end{equation}
$\sfz$ is $1$-Lipschitz and differentiable a.e.; moreover,
 $\sfz$ is differentiable in $\Omega$ with
\begin{equation}
  \label{eq:15}
  \sfz'(t)=\vartheta'(\mathsf x(t))\mathsf x'(t)=1  \quad
  \text{for
every $t\in \Omega$},\quad
\sfz'(t)=0\quad\text{a.e.~in }[0,T_\star)\setminus \Omega
\end{equation}
(see e.g. [\cite{leoni2009first}, Theorem 3.44] for the chain rule for absolutely continuous functions).\\
We will denote by $ {\mathsf t}:[0,\Theta)\to [0,T_\star)$ the
continuous inverse map of
$\sfz$ which is differentiable in the dense open set $\sfz(\Omega)$
with derivative $1$. Since ${\mathsf t}$ is increasing, it is of
bounded variation in every compact interval $[0,\Theta']$ with $\Theta'<\Theta$. 
For every $h\in \H$ 
we set $u_h(t):=\langle u(t),h\rangle$,
$w:=u\circ {\mathsf t}:[0,\Theta)\to \H$, and $w_h:=u_h\circ {\mathsf t}:[0,\Theta)\to \R $.
{\normalcolor Since $u_h$ is locally Lipschitz}, $w_h$ is a function of bounded variation 
in every compact interval $[0,\Theta']$ with $\Theta'<\Theta$: we want to show that $w_h$ is
absolutely continuous in $[0,\Theta']$.
To this aim, we use the chain rule for \textrm{BV} functions (see
e.g. [\cite{AmbrosioFuscoPallara00}, Theorem 3.96])
and the facts that $u_h,{\mathsf t}$ are continuous, {\normalcolor $u_h$ is Lipschitz in $[0, \mathsf t(\Theta')]$,} and that ${\mathsf t} $
is continuously differentiable on the open set 
$\sfz(\Omega)$; the Cantor part $\mathrm D^c\, {\mathsf t} $ of the distributional derivative of
${\mathsf t}$ is therefore concentrated on the set $(0,\Theta')\setminus
\sfz(\Omega)$ and the BV chain rule yields
\begin{equation}
  \label{eq:16}
  \mathrm D^c \,w_h=(u_h'\circ {\mathsf t})\mathrm D^c\,
  {\mathsf t}\quad
  \text{where } u_h'(t):=\langle u'(t),h\rangle=\langle
  -\nabla\phi(u(t)),h\rangle=
  -\nabla_h\phi(u(t)).
\end{equation}
On the other hand, for every $s\in (0,\Theta')\setminus \sfz(\Omega)$ we
have
$\mathsf t (s)\in (0, T_\star)\setminus \Omega$ and thus 
$\nabla\phi(u(\mathsf t(s)))=0$. We conclude that $\mathrm D^c
\,w_h=0$ and $w_h$ is {\normalcolor locally absolutely continuous}. 
The same argument shows that the pointwise derivative of $w_h$
vanishes a.e.~in $(0,\Theta)\setminus \sfz(\Omega)$, whereas 
the computation of the derivative of $w$ in $\sfz(\Omega)$ yields
$$w'(s)=u'({\mathsf t}(s)){\mathsf t}'(s)=
u'({\mathsf t}(s))=
-\nabla\phi(u({\mathsf t}(s)))=-\nabla\phi(w(s))$$
Summarizing, we obtain
\begin{equation}
  \label{eq:20}
  w_h'(s)=-\nabla_h\phi(w(s))\quad \text{a.e.~in }(0,\Theta);
\end{equation}
since the righthand side of \eqref{eq:20} is continuous we deduce that
$w_h$ is a $\mathrm C^1$ function and \eqref{eq:20} holds in fact
everywhere in $[0,\Theta)$. Being $w$ continuous and scalarly $\mathrm
C^1$, we deduce that $w$ is of class $\mathrm C^1$ in $[0,\Theta)$ 
and $w$ is a solution of \eqref{eq:1} satisfying $w(s)=u(\mathsf
t(s))$.
If $\Theta$ is finite, the uniform H\"older estimate \eqref{eq:5} shows that
$w$ admits the limit $\bar
w:=\lim_{s\uparrow \Theta}w(s)=\lim_{t\uparrow+\infty}u(t)$. It follows that $\bar w$
is a stationary point of $\phi$, so that extending $w$ by the
constant value $\bar w$ for $t\ge \Theta$ still yields a solution to
\eqref{eq:1}. If we have $T_\star < \infty$, we can extend $\sfz$ by the constant value $\Theta = \lim_{t\uparrow T_\star}\sfz(t) < \infty$ for $t\ge T_\star$. 
Since we have $\range w\subset \overline{\range u}$ and 
$u(t)=w(\sfz(t))$ for every $t\ge 0$, we deduce that $u\succ w$.
Since $u$ is \minimal, we should have $w\equiv u$ so that
$\sfz(t)\equiv t$ for $t\in[0, T_\star)$. \eqref{eq:15} then yields that $[0,T_\star)\setminus
\Omega$ has $0$ Lebesgue measure and \eqref{eq:7} holds.

\medskip\noindent
(6) If $u\in \minsolution\phi$ and $T_\star(u) > 0$, we know that the map $\varphi:t\mapsto \phi(u(t))$ is
of class $\mathrm C^1$, strictly decreasing with $\varphi'(t)<0$
a.e.~in $(0,T_\star)$. It follows that it has a locally absolutely continuous
inverse $\psi$. Conversely, if $\varphi$ has a locally absolutely continuous
left inverse $\psi$ (which is then also the inverse) then $\varphi'(t)=-|\nabla\phi(u(t))|^2\neq 0$ a.e.~in
$(0, T_\star)$, 
so that \eqref{eq:7} holds and $u\in \minsolution\phi$ by the previous
claim (5).
\end{proof} 
We conclude this section with a definition and a simple remark.
\begin{definition}[Eventually \minimal\ solutions]
  \label{def:emaxsolution}
  We say that a solution $u\in \solution\phi$ is \emph{eventually
    \minimal} if there exists a time $T>0$ such that $u'(T)\neq0$ and the curve $t\mapsto u(t+T)$ is a \minimal\ non-constant solution.
\end{definition}
\begin{remark}[Approximation by eventually \minimal\ solutions]
  \label{rem:emaxsolution}
  \upshape
  Any non-constant $u\in \solution\phi$ may be locally uniformly
  approximated by a sequence of eventually \minimal\ solutions keeping
  the same initial data. For every $n\in \N$ it is sufficient to
  choose an increasing sequence $t_n\uparrow T_\star(u)$ with
  $u'(t_n)\neq0$ and replace
  the curve $v_n:=u(\cdot+t_n)$ with the unique
  \minimal\ solution $w_n$ such that $v_n\succ w_n$, given by Theorem
  \ref{thm: max gf}. The curves
  \begin{equation}
    \label{eq:60}
    u_n(t):=
    \begin{cases}
      u(t)&\text{if }0\le t\le t_n,\\
      w_n(t-t_n)&\text{if }t>t_n.
    \end{cases}
  \end{equation}
  are eventually \minimal\ and converge to $u$ uniformly on compact intervals.
  
  Any constant $u\in\solution\phi$ is minimal. 
\end{remark}
\Minimal\ gradient flows will play a crucial role in the proof of De Giorgi's conjecture. Roughly speaking, the conjecture can be proved directly for this class of gradient flows, and in addition, any other gradient flow can be approximated by a sequence of \minimal\ gradient flows.

\section{Approximation of the \minimal\ gradient flow}\label{sec: approx of max gf}

In this section we study a particular family of perturbations that
will be extremely useful to approximate \minimal\ gradient flows.
As a first step, we present a general strategy to force a discrete
solution
of the minimizing movement scheme to stay in a prescribed compact set.
We will always assume that $\phi\in \rmC^1(\H)$
satisfies the uniform quadratic bound \eqref{eq:coercive intro}, so that
\begin{equation}
  \label{eq:108}
  \inf_{y\in \H}\frac 1{2\tau}|x-y|^2+\phi(y)>-\infty\quad
  \text{for every }x\in \H,\ \tau\in (0,\tau_*).
\end{equation}

\subsection{Distance penalizations from compact sets}\label{subsec: 4.1}
\label{subsec:penalization}
Let a time step $\tau > 0$ and a nonempty compact set $\cU\subset \H$ be
fixed. We denote by $\psi_\cU:\H\to \R$ the distance function
\begin{equation}
  \label{eq:8}
  \psi_\cU(x):=\operatorname{dist}(x,\cU)=
  \min_{y\in \cU}|x-y|, 
\end{equation}
 by $\Gamma_\cU$ the closed convex set 
 \begin{equation}
   \label{eq:30}
   \Gamma_\cU:=\Big\{(a,b)\in
   [0,\infty)\times [0,\infty): |\nabla\phi(x)-\nabla\phi(y)|\land 1\le a+b|x-y|\ \text{for every }x\in \cU,\ y\in
   \H \Big\}
 \end{equation}
and by $\omega_\cU:[0,\infty)\to [0,\infty)$ the concave modulus of continuity
\begin{equation}
  \label{eq:18}
  \omega_\cU(r) := \inf \Big\{a+br:(a,b)\in \Gamma_\cU\Big\}.  
\end{equation}
Notice that
\begin{equation}
\text{$\omega_\cU$ is increasing, bounded by $1$,
concave, and satisfies \quad
$\lim_{r\downarrow0}\omega_\cU(r)=0$,}
\label{eq:53}
\end{equation}
with
\begin{equation}
  \label{eq:107}
  |\nabla\phi(x)-\nabla\phi(y)|\land 1\le \omega_\cU(|x-y|)\quad
  \text{whenever }x\in \cU,\ y\in \H.
\end{equation}
In order to prove the limit property of \eqref{eq:53}, we can argue by contradiction;
let us assume that we have instead $\inf_{r>0} \omega_\cU(r)=\bar
a\in (0,1]$. 
Choosing $r=\bar a/(4n)$, $n\in \N$, we see that the couple $(\bar a/2,n)$ does
not belong to $\Gamma_\cU$, so that 
for every $n\in \N$ there exist $x_n\in \cU$ and $y_n\in \H$ such that 
\begin{equation}\normalcolor
1\land |\nabla\phi(x_n)-\nabla\phi(y_n)|-n|x_n-y_n|>\bar
a/2.\label{eq:31}
\end{equation}
In particular $\normalcolor |x_n-y_n|\le 1/n$ so that
$\lim_{n\to\infty}|x_n-y_n|=0$.
Since $x_n\in \cU$ and $\cU$ is compact, we can extract a subsequence
$k\mapsto n(k)$ such that $\lim_{k\to\infty}x_{n(k)}=x\in \cU$, and thus 
$\lim_{k\to\infty}y_{n(k)}=x$ as well and therefore 
$\lim_{k\to\infty}|\nabla\phi(x_{n(k)})-\nabla\phi(y_{n(k)})|=0$ by
the continuity of $\nabla \phi$, a contradiction with
\eqref{eq:31}.

We consider a family of perturbations of the
function $\phi$ 
depending on a parameter $\lambda\ge0$ and on a compact set
$\cU\subset \H$. It is given by
\begin{gather}
  \label{eq:21}
  \vphi\lambda\cU x:=\phi(x)+\lambda\,\psi_\cU(x), \quad
  \vPhi\lambda\cU{\tau,x,y}:=\frac 1{2\tau}|x-y|^2+\vphi\lambda\cU y,
  \\
  \label{eq:22}
  \vY{\tau}\lambda\cU x:=\argmin\vPhi\lambda\cU{\tau,x,\cdot}.
\end{gather}
Our aim is to give a sufficient condition on the choice of $\lambda$ in dependence of $\tau$ and $\cU$ in order to be sure that 
whenever $x\in \cU$ the minimizing set $\vY{\tau}\lambda\cU x$
is nonempty and it is contained in $\cU$ as well.

{\normalcolor In Lemma \ref{lem:def of tau_U}, a rough estimate of $|\nabla\phi(y)|$ of an approximate minimizer $y$ of $\vPhi\lambda\cU{\tau,x,\cdot}$ is given. 
\begin{lemma}\label{lem:def of tau_U}
There exists $\tau_{\cU}\in(0, \tau_*)$ so that for every $y\in\H, \ x\in\cU, \ \tau\in(0, \tau_\cU)$ satisfying 
\begin{equation}\label{eq:am tau_U}
\phi(y) + \frac{1}{2\tau} |x-y|^2  \leq \phi(x) + |x-y|,
\end{equation}
it holds that 
\begin{equation}\label{eq:def of tau_U}
|\nabla\phi(y) - \nabla\phi(x)| \le \frac{1}{2}. 
\end{equation}
\end{lemma}
\begin{proof}
Since $\lim_{r\downarrow0}\omega_\cU(r)=0$, there exists $\bar{r} > 0$ such that $\omega_\cU(r) \le \frac{1}{2}$ for every $0\le r < \bar{r}$. In view of (\ref{eq:107}), it is sufficient to prove that there exists $\tau_\cU\in(0, \tau_*)$ such that $|x-y| < \bar{r}$ whenever $y\in\H, \ x\in\cU$ satisfy (\ref{eq:am tau_U}) for some $\tau\in(0, \tau_\cU)$.

Let us suppose that (\ref{eq:am tau_U}) holds for $y\in\H, \ x\in\cU, \ \tau\in(0, \tau_*)$. We apply [\cite{AGS08}, Lemma 2.2.1] and (\ref{eq:coercive intro}) in order to obtain 
\begin{eqnarray*}
|x-y|^2 &\leq& \frac{4\tau\tau_*}{\tau_* - \tau} \left(\phi(y) + \frac{1}{2\tau} |x-y|^2 + \phi_* + \frac{1}{\tau_* - \tau} |x|^2\right) \\
&\leq& \frac{4\tau\tau_*}{\tau_* - \tau} \left(\max_{z\in\cU} \phi(z) + \frac{1}{2} + \frac{1}{2} |x-y|^2 + \phi_* + \frac{1}{\tau_* - \tau} \max_{z\in\cU}|z|^2\right).
\end{eqnarray*}
The claim now easily follows.   
\end{proof}
}

The following Lemma \ref{lem:standard-result} is a typical result
for nonsmooth analysis of the distance function.
\begin{lemma}\label{lem:standard-result}
  Let $\LL\cU:=\max_\cU|\nabla\phi|\lor 1$, 
  $\normalcolor 0\le \eta\le \lambda< 1/4, \ \tau\in(0, \tau_\cU), \ x\in\cU$ 
  and $y\in \H$ be an approximate $\eta$-minimizer of
  $\vPhi\lambda\cU{\tau,x,\cdot}$, i.e.
  \begin{equation}
    \label{eq:24}
    \vPhi\lambda\cU{\tau,x,y}\le 
    \vPhi\lambda\cU{\tau,x,w}+\eta|w-y|\quad\text{for every }w\in \H.
  \end{equation}
  Then the vector $\displaystyle \xi:=\frac{y-x}\tau+\nabla \phi(y)$ satisfies
  \begin{equation}
    \label{eq:12}
    |\xi|\le \lambda+\eta,\quad 
  \normalcolor  |y-x|\le (\LL\cU+1/2 + \lambda+\eta)\tau\le 2\LL\cU\tau.
  \end{equation}
  Moreover, if $y\not\in \cU$, 
  then $|\xi|\ge \lambda-\eta$.
\end{lemma}
\begin{proof}
  Since $\psi_\cU$ is $1$-Lipschitz, the minimality condition
  \eqref{eq:24} yields for every $w\in \H$
  \begin{displaymath}
    \phi(w)+\frac 1{2\tau}|x-w|^2-\phi(y)-\frac1{2\tau}|x-y|^2\ge
    \lambda \psi_\cU(y)-\lambda\psi_\cU(w)-\eta|w-y|
    \ge -(\lambda+\eta) |y-w|.
  \end{displaymath}
  We can choose $w:=y+\theta v$, divide the above inequality by $\theta>0$ and pass to the limit as $\theta\downarrow 0$ obtaining
  \begin{displaymath}
    \langle \xi,v\rangle\ge -(\lambda + \eta) |v|\quad \text{for every $v\in \H$},
  \end{displaymath}
  which yields {\normalcolor the first part of \eqref{eq:12}. The second part of \eqref{eq:12} then follows from the estimate $|y-x| \leq \tau(|\xi|+|\nabla\phi(y) - \nabla\phi(x)| + |\nabla\phi(x)|)$ and \eqref{eq:def of tau_U}.}
  
  If we choose $w:=(1-\theta)y+\theta \hat{y}$ with $\hat{y}\in \cU$
  satisfying $|y-\hat y|=\psi_\cU(y)>0$, we also obtain
  $\psi_\cU(w)=|(1-\theta)y+\theta
  \hat{y}-\hat{y}|=(1-\theta)|y-\hat{y}|$ and
  $|y-w|=\theta |y-\hat y| $ so that 
  \begin{displaymath}
    \phi(w)+\frac 1{2\tau}|x-w|^2-\phi(y)-\frac1{2\tau}|x-y|^2\ge
    \lambda\Big(\psi_\cU(y)-\psi_\cU(w)\Big)-
    \eta|y-w|=
    \theta (\lambda-\eta) |y-\hat{y}|
  \end{displaymath}
  and therefore
  \begin{displaymath}
    \langle \xi,\hat{y}-y\rangle\ge (\lambda-\eta) |y-\hat{y}|
  \end{displaymath}
  which yields $|\xi|\ge \lambda-\eta$.  
\end{proof}
%
The next lemma provides a suitable condition on the choice of $\lambda$.  
\begin{lemma}\label{lem: choice of lambda}
  Let $\cU$ be a compact subset of $\H$,
  $\LL\cU:=\max_\cU|\nabla\phi|\lor 1$,
  $x, z\in \cU$, $\normalcolor \tau\in (0,\tau_\cU)$, and $\normalcolor \lambda,\delta\in [0,1/4)$,
  satisfy 
  \begin{equation}
    \label{eq:19}
    \Big|\frac{z-x}\tau+\nabla\phi(z)\Big|\le \delta,
  \end{equation} 
\begin{equation}
    \label{eq:23}
    \lambda^2> \fourteen\, \LL\cU\, \omega_\cU(3\LL\cU\tau)+
    2\delta^2.
  \end{equation}
  Then $\vY\tau\lambda\cU x$ is nonempty and contained in
  $\cU$.
\end{lemma}
\begin{proof}
  We argue by contradiction and we suppose that 
  \begin{equation}
    \text{there exists $y\in
      \H\setminus \cU$ such that}\quad
    \vPhi\lambda\cU{\tau,x,y}\le \min_{u\in
      \cU}\vPhi\lambda\cU{\tau,x,u}.\label{eq:28}
\end{equation}
  We can apply Ekeland variational principle in $\H$ to the continuous
  function
  \begin{displaymath}
   \normalcolor w\mapsto \Phi_{\lambda,\cU}(\tau,x,w)
  \end{displaymath}
  which is bounded from below by \eqref{eq:108}.
  For every $\eta>0$ we
  can find $y_\eta\in \H$ satisfying the properties
  \begin{gather}
    \label{eq:25}
    \vPhi\lambda\cU{\tau,x,y_\eta}+\eta|y_\eta-y|\le
    \vPhi\lambda\cU{\tau,x,y},\\
    \label{eq:26}
    \vPhi\lambda\cU{\tau,x,y_\eta}\le
    \vPhi\lambda\cU{\tau,x,w}+\eta|y_\eta-w|
    \quad\text{for every }w\in \H.
  \end{gather}
  \eqref{eq:25} and \eqref{eq:28} yield that $y_\eta\not\in \cU$
  and
  \begin{equation}
  \label{eq:27}
  \phi(y_\eta)+\frac1{2\tau}|y_\eta-x|^2+\lambda \psi_\cU(y_\eta)\le \phi(z)+\frac1{2\tau}|z-x|^2.
  \end{equation}
  Choosing $\eta$ sufficiently small so that 
  $\normalcolor \lambda+\delta+\eta\le 1/2$, \eqref{eq:12} and \eqref{eq:19} yield
  \begin{equation}
    \label{eq:109}
    \normalcolor |y_\eta-x|\le (\LL\cU + 1/2 + \lambda+\eta)\tau\le 2\LL\tau\,\tau,\quad
    |z-x|\le
    (\LL\cU +\delta)\tau,
  \end{equation}
  and therefore
  \begin{equation}
    \label{eq:110}
  \normalcolor  |y_\eta-z|\le (2\LL\cU + 1/2 + \lambda+\delta+\eta)\tau\le 3\LL\cU \tau.
  \end{equation}
  Since $\omega_\cU(3\LL\tau\,\tau)<\lambda^2\le 1$ by \eqref{eq:23},
  we get
  the estimate
  \begin{align*}
    \big|\nabla\phi((1-t)y_\eta+t
      z)-\nabla\phi(y_\eta)\big|
    &\le 
      \big|\nabla\phi((1-t)y_\eta+t
      z)-\nabla\phi(z)\big|+\big|\nabla\phi(z)-\nabla\phi(y_\eta)\big|
     \\& \le 2\omega_\cU(|y_\eta-z|)\quad
         \text{for every }t\in [0,1].
  \end{align*}
  The integral mean value Theorem
  \begin{align*}
    \phi(z)-\phi(y_\eta)-\langle \nabla\phi(y_\eta), z-y_\eta\rangle
    &=
      \int_0^1 \langle \nabla\phi((1-t)y_\eta+t
      z)-\nabla\phi(y_\eta),z-y_\eta\rangle\,\rmd t
  \end{align*}
  yields
	\begin{equation}
          \label{eq:29}
          \big|\phi(z)-\phi(y_\eta)-\langle \nabla\phi(y_\eta),z-y_\eta\rangle\big|
          \le 2|z-y_\eta|\omega_\cU(|z-y_\eta|).
  \end{equation}
  So, combining \eqref{eq:27} and \eqref{eq:29} we obtain
  \begin{align*}
    \frac1{2\tau}|y_\eta-x|^2-\frac1{2\tau}|z-x|^2-\langle \nabla
    \phi(y_\eta),z-y_\eta\rangle+
    \lambda \psi_\cU(y_\eta)&\le 
    \phi(z)-\phi(y_\eta)- \langle \nabla
                            \phi(y_\eta),z-y_\eta\rangle
                            \\&\le 2|z-y_\eta|\omega_\cU(|z-y_\eta|).
  \end{align*}
  Using the identity $|a|^2-|b|^2=\langle a+b,a-b\rangle$ and
  neglecting the positive term $\lambda \psi_\cU(y_\eta)$ we get
  \begin{equation*}
    \frac 1{2\tau} \langle y_\eta-x+2\tau \nabla\phi(y_\eta)+z-x,y_\eta-z\rangle\le 
    2|z-y_\eta|\omega_\cU(|z-y_\eta|).
  \end{equation*}
  Setting $\xi_\eta:= \frac{y_\eta-x}\tau+\nabla\phi(y_\eta)$ as in
  Lemma \ref{lem:standard-result} we get 
  \begin{equation*}
    y_\eta-z=y_\eta-x+x-z=
    \tau \xi_\eta-\tau\nabla \phi(y_\eta)+x-z.
  \end{equation*}
Thus, we obtain
  \begin{equation*}    
  \frac 1{2\tau}\langle \tau\xi_\eta+\tau\nabla\phi(y_\eta)+z-x,\tau\xi_\eta-\tau\nabla \phi(y_\eta)-(z-x)\rangle
    \le 
    2|z-y_\eta|\omega_\cU(|z-y_\eta|), 
  \end{equation*}
yielding
  \begin{equation*}
    \frac\tau2|\xi_\eta|^2\le 
    \frac 1{2\tau}|\tau\nabla\phi(y_\eta)-(x-z)|^2+
    2|z-y_\eta|\omega_\cU(|z-y_\eta|).
  \end{equation*}
  Using \eqref{eq:19} and the fact that $|\xi_\eta|\ge \lambda-\eta$
  if $\eta\le \lambda$ 
  by Lemma \ref{lem:standard-result}, we obtain
  \begin{align*}
    |\lambda-\eta|^2 &\le 
                       2\Big(|\nabla\phi(y_\eta)-\nabla\phi(z)|^2+\delta^2\Big)+
                       \frac 4\tau |z-y_\eta|\omega_\cU(|z-y_\eta|)
    \\&\le 
        2\Big(\omega_\cU^2(3\LL\cU \tau)+\delta^2\Big)+
        12\LL\cU  \,\omega_\cU(3\LL\cU \tau)
        \le \fourteen \LL\cU  \,\omega_\cU(3\LL\cU \tau)+ 2\delta^2,
  \end{align*}
  where we used \eqref{eq:110} 
  and the fact that $\omega_\cU\le 1$.
  Since $\eta$ can be chosen arbitrarily small, 
  we get a contradiction with \eqref{eq:23}.
\end{proof}
Notice that the use of Ekeland variational principle in the previous
proof is only needed when $\H$ has infinite dimension. 
If $\H$ has finite dimension, one can directly select $y_\eta$ as the minimizer
of $\Phi_{\lambda,\cU}(\tau,x,\cdot)$ in $\H$ setting $\eta=0$.
\begin{corollary}
  \label{cor:invariance}
  Let $\cU\subset \H$ be a compact set, $\LL\cU:=1\lor 
  \max_\cU |\nabla\phi|$, $\normalcolor \lambda,\delta\in [0,1/4)$,
  $\normalcolor \tau\in (0,\tau_\cU)$.
  If \eqref{eq:23} holds and
  for every $x\in \cU$ there exists $z\in \cU$ satisfying \eqref{eq:19},
 then for every initial choice of
 $u_0\in \cU$ the set $\MS\tau{\varphi_{\lambda,\cU}}{u_0}$ is nonempty and every discrete solution $U\in
   \DMM\tau{\varphi_{\lambda,\cU}}{u_0} $ takes values in $\cU$.
\end{corollary}
\subsection{Strong approximation of \minimal\ solutions}\label{subsec: 4.2}
We can now apply Lemma \ref{lem: choice of lambda} and Corollary \ref{cor:invariance} in order to construct good discrete
solutions by choosing suitable compact subsets of the range of
$u\in \solution\phi$. 
We distinguish two cases:
the next lemma contains the fundamental estimates  
in the case when $\phi$ is bounded on the range of a solution $u$;
Lemma \ref{lem:basic estimate2} 
will deal with solutions $u$ for which $\phi(u(t))\to-\infty$ as
$t\to+\infty$.

We introduce the following notation (recall Remark
\ref{rem:bounded-interval}): 
if $u\in \solution\phi$, $T>0$,
$\tau>0$ we set
  \begin{equation}
    \label{eq:55}
    \mathcal U(\tau, T):=\{u(n\tau): 0\le n \le \nn\tau T\}.
  \end{equation}
%
\begin{lemma}\label{lem:basic estimate}
  Let $u\in \solution\phi$ such that 
  \begin{equation}
    \label{eq:87}
    \inf_{t\ge 0}\phi(u(t))=\lim_{t\uparrow\infty}\phi(u(t))>-\infty.
  \end{equation}
  For every $\normalcolor \eps\in (0,1/4)$ there exist
  $T=T(\eps)\ge \eps^{-1}$ and $\bar{\tau}=\bar{\tau}(\eps)\in (0,1)$
  such that
  for every $0 < \tau \le \bar{\tau}$ 
  the set $\DMM\tau{\varphi_{\eps,\cU(\tau,T)}}{u(0)}$ is nonempty, 
  {\normalcolor every element $U\in \DMM\tau{\varphi_{\eps,\cU(\tau,T)}}{u(0)}$ 
  takes values in $\cU(\tau,T)\subset u([0,T+1])$} and satisfies
  \begin{equation}\label{eq: 57 2}
  \phi(U(t))\le \phi(u(t\land T))\quad\text{for every }t\ge 0. 
  \end{equation}
  \normalcolor Moreover, for every $S > 0$, it holds that
  \begin{equation}\label{eq:mm ext}
  \DMM\tau{\varphi_{\eps,\cU(\tau,T)}}{u(0),S} = \{U|_{[0, S]} \ | \ U\in \DMM\tau{\varphi_{\eps,\cU(\tau,T)}}{u(0)}\}. 
  \end{equation}
\end{lemma}
\begin{proof}
Since $u$ satisfies \eqref{eq:87}, the identity \eqref{eq:4} 
  yields $\int_0^\infty |\nabla\phi(u(t))|^2\,\rmd t<\infty$ and therefore
  \begin{equation}\label{eq: liminf is zero}
    \liminf_{t\uparrow\infty} |\nabla\phi(u(t))| = 0.
  \end{equation}
{\normalcolor We select $T\ge \eps^{-1}$ such that $|\nabla\phi(u(T))|\le
  \eps/4$ and 
  consider the compact set $\mathcal K:=u([0,T+1])$;}
  notice that $\cU(\tau,T)\subset \cK$ for every $\tau\le 1$.

  We set $L:=1\lor \max_\cK|\nabla\phi|$ and
  we choose $\delta:=\eps/2$ and 
  {\normalcolor $\bar\tau<\tau_\mathcal K \land 1$ (with $\tau_\mathcal K$ as in Lemma \ref{lem:def of tau_U})} so that 
$(\fourteen \,L+1)\omega_\cK(3L\bar \tau)<\eps^2/2$;
in particular
\begin{equation}
  \label{eq:88}
  \fourteen \,L\omega_\cK(3L\bar \tau) +2\delta^2
  <\eps^2,\quad
  \omega_\cK(L\bar \tau)\le\delta/2.
\end{equation}
  We observe that 
  for every $x=u((n-1)\tau)\in \cU$, $1\le n\le N$, $N=\nn\tau T$, $\normalcolor \tau\in (0, \bar{\tau}]$,
  the choice $z:=u(n\tau)$
  satisfies \eqref{eq:19} since 
  \begin{align}
    \notag\frac{z-x}\tau+\nabla\phi(z)
    &=
      \frac{u(n\tau)-u((n-1)\tau)}\tau+\nabla\phi(u(n\tau))
      \\
    \label{eq:91}
    &=\frac1\tau\int_{(n-1)\tau}^{n\tau} \Big(\nabla\phi(u(n\tau))-\nabla\phi(u(r))\Big)              \,\mathrm d r  
  \end{align}
  and therefore
  \begin{equation}
    \label{eq:90}
    \left|    \frac{z-x}\tau+\nabla\phi(z)\right|\le \omega_{\mathcal
      K}(L\tau)\le \omega_{\mathcal
      K}(L\bar \tau)\le \delta/2
  \end{equation}
  by \eqref{eq:88}. Notice that $|u'(t)|=|\nabla\phi(u(t))|\le L$ for
  $t\in [0,T+1]$ so that $|u(n\tau)-u(r)|\le L\tau$ whenever $r\in
  ((n-1)\tau,n\tau]$.
  
  For $x=u(N\tau)$ we can choose $z=x=u(N\tau)$, since in this case
  \begin{displaymath}
    |\nabla\phi(z)|\le |\nabla\phi(z) - \nabla\phi(u(T))| + |\nabla\phi(u(T))| \le \omega_{\mathcal{K}}(L\tau) + \frac{\delta}{2} \le \delta. 
  \end{displaymath}
  
  Since $\omega_\cU(r)\le \omega_\cK(r)$, we can apply Lemma \ref{lem:
    choice of lambda} with the choice $\lambda:=\eps$ 
  thanks to \eqref{eq:88}: we
  obtain the fact that $\DMM\tau{\varphi_{\eps,\cU(\tau,T)}}{u(0)}$ is nonempty, every element $U\in \DMM\tau{\varphi_{\eps,\cU(\tau, T)}}{u(0)}$ 
  takes values in $\cU(\tau,T)$ {\normalcolor and (\ref{eq:mm ext}) holds}.

  In order to prove \eqref{eq: 57 2} we write
  $U(t)=\sum_{n}U^n_\tau\nchi(t/\tau-(n-1))$ for $t > 0$ and we observe that
  \eqref{eq: 57 2} is equivalent to 
  \begin{equation}
    \label{eq:58}
    \phi(U^n_\tau)\le \phi(u(n\tau\land T))\quad\text{for every }n\in \N
  \end{equation}
  thanks to the monotonicity of $t\mapsto \phi(u(t))$.

  We argue by induction, observing that \eqref{eq:58} is clearly true
  for $n=0$. 
  
  If  
  $\phi(U^{n-1}_\tau)\le \phi(u((n-1)\tau))$ for some $1\le n\le N$, then  
  we can deduce that
  $U^{n-1}_\tau=u(k\tau)$ for some $k\ge n-1$.

  If $k>n-1$ then we easily get $\phi(U^n_\tau)\le
  \phi(U^{n-1}_\tau)\le \phi(u(k\tau))\le \phi(u(n\tau))$.

    It remains to consider the case $k=n-1$,
  i.e. $U^{n-1}_\tau=u((n-1)\tau)$. If $\phi(u(n\tau))=\phi(u((n-1)\tau))$, the induction step is
  obvious. If $\phi(u(n\tau))<\phi(u((n-1)\tau))$, then 
  it is sufficient to observe that $\vPhi\lambda\cU{\tau,u((n-1)\tau),
    u(n\tau)} 
  < \phi(u((n-1)\tau))$. Indeed, it then holds {\normalcolor by \eqref{eq:4}} that 
\begin{align*}
  \phi(u(n\tau))&+\frac{1}{2\tau} |u(n\tau) - u((n-1)\tau) |^2 
   \leq
     \phi(u(n\tau))+\frac{1}{2}\int_{(n-1) \tau}^{n\tau}{|u'(r)|^2 \, \mathrm d r}
  \\&= 
      \phi(u(n\tau))+ \frac{1}{2}\big(\phi(u((n-1)\tau)) -
      \phi(u(n\tau))\big)
    \\&  =
       \phi(u((n-1)\tau))-
       \frac{1}{2}\big(\phi(u((n-1)\tau)) - \phi(u(n\tau))\big)
       <\phi(u((n-1)\tau)),
\end{align*}
so that $U^n_\tau$ belongs to $\{u(k\tau):n\le k\le N\}$ 
and thus satisfies
$\phi(U^n_\tau)\le \phi(u(n\tau))$. 

Eventually, for $n> N$, the induction step is trivial.  
 \end{proof}
\begin{remark}
  \label{rem:liminf}
\upshape 
\normalcolor 
The proof shows that the statement of Lemma \ref{lem:basic estimate} in fact holds for every $u\in\solution\phi$ satisfying 
\eqref{eq: liminf is zero}.
\end{remark}
We now consider the case when $\phi$ is unbounded on $\range u$.
\begin{lemma}\label{lem:basic estimate2}
  Let $u\in \solution\phi$ such that 
  \begin{equation}
    \label{eq:89}
    \inf_{t\ge 0}\phi(u(t))=\lim_{t\uparrow\infty}\phi(u(t))=-\infty.
  \end{equation}
  For every $\normalcolor \eps\in (0,1/4), \ T > 0$ there exist
  $\bar \tau=\bar\tau(\eps,T)\in (0,1)$ and $\normalcolor \bar T=
  \bar T(T)\ge T$
  such that 
  for every $0 < \tau \le \bar{\tau}$ 
  the set $\DMM\tau{\varphi_{\eps,\cU(\tau,\bar T)}}{u(0),T}$ is nonempty,
  every element $\normalcolor U\in \DMM\tau{\varphi_{\eps,\cU(\tau,\bar
      T)}}{u(0),T}$ 
  takes values in $\cU(\tau,{\bar T})
  \subset u([0,\bar T+1])$ 
  and satisfies
  \begin{equation}\label{eq: 57 2bis}
    \phi(U(t))\le \phi(u(t))\quad\text{for every }t\in [0,T]. 
  \end{equation}
\normalcolor Moreover, for every $0\le S \le T$, it holds that
\begin{equation}\label{eq:mm ext 2}
\DMM\tau{\varphi_{\eps,\cU(\tau,\bar
      T)}}{u(0),S} = \{U|_{[0, S]} \ | \ U\in\DMM\tau{\varphi_{\eps,\cU(\tau,\bar
      T)}}{u(0),T}\}. 
\end{equation}
\end{lemma}
\begin{proof}
The argument of the proof is quite similar to the one 
of Lemma \ref{lem:basic estimate}: the only difference is that 
we cannot find a compact set containing the range of 
the whole discrete
solutions. 

Let us set $F:=\phi(u(0))\lor |u(0)|^2$ and let
$C=C(\phi_*,\tau_*,F,T)$ the constant provided by Lemma \ref{le:apriori}(ii).
By \eqref{eq:89} we can select a time $\bar T\ge T$ such that 
\begin{equation}
  \label{eq:82}
  \phi(u(\bar T))<\phi(u(0))-C
\end{equation}
and we set $\cK:=u([0,\bar T+1])$, 
$L:=1\lor \max_\cK|\nabla\phi|$,
$\bar N=\nn\tau{\bar T}$,
$\delta=\eps/2$ and $\normalcolor \bar \tau\in (0,1\land\tau_*/16\land\tau_\cK)$ 
sufficiently small so that 
\eqref{eq:88} holds.

Since $\cU(\tau,\bar T)\subset \cK$, 
the same calculations of \eqref{eq:91} and \eqref{eq:90}
show that for every $x\in \{u(k\tau): 0\le k <\bar N\}$ 
there exists $z\in \cU(\tau,{\bar T})$ satisfying \eqref{eq:19}.

We can then apply Lemma \ref{lem: choice of lambda}
and the same induction argument 
of the previous proof to prove that an integer $M\ge 1$ 
and a sequence
$\normalcolor (U^n_\tau)_{0\le n\le M}\in \MS\tau{\varphi_{\eps,\cU(\tau,\bar T)}}{u(0),M}$
exist such that 
{\normalcolor$U^n_\tau\in \{u(k\tau): 0\le k \le\bar N\}$}
and $U^M_\tau=u(\bar N\tau)$.
Since $\phi(U^M_\tau)=\phi(u(\bar N\tau))\le 
\phi(u(\bar T))< \phi(u(0))-C$ and \eqref{eq:111} yields
\begin{equation}
  \label{eq:114}
  \phi(U^n_\tau)\ge \phi(u(0))-C
  \quad
  \text{for every $1\le n\le \nn\tau T$},
\end{equation}
we deduce that $\nn\tau T<M$
so that
$\DMM\tau{\varphi_{\eps,\cU(\tau,\bar T)}}{u(0),T}$ is not empty. 

If now $U$ is any element of $\DMM\tau{\varphi_{\eps,\cU(\tau,\bar
    T)}}{u(0),T}$ corresponding to 
a sequence $(U^n_\tau)_{0\le n\le N}
\in \MS\tau{\varphi_{\eps,\cU(\tau,\bar T)}}{u(0),N}$,
$N=\nn\tau T$, {\normalcolor then Lemma \ref{lem: choice of lambda},} the same induction argument of the previous proof and \eqref{eq:114}
show that $U$ take values in $\cU(\tau,
\bar T)$ and \eqref{eq: 57 2bis} holds. {\normalcolor The same arguments show that \eqref{eq:mm ext 2} holds for every $0\le S \le T$.}
\end{proof}

We are now able to state the main result of this section. 
\begin{theorem}
  \label{thm:appmax}
  Every \minimal\ solution $u\in \minsolution\phi$ is strongly
  approximable in every compact interval, 
  according to Definition \ref{def:PAGF}. 

  If in addition $\H$ has finite dimension or \eqref{eq: liminf is zero} is satisfied, then $u$ is strongly approximable according to Definition \ref{def:AGF}. 
\end{theorem}
\begin{proof}\normalcolor
  We pick a decreasing sequence $\eps_n\downarrow0$ 
  and an increasing sequence $T_n:=\eps_n^{-1}\uparrow+\infty$.
  
  If \eqref{eq:87} holds, we can apply Lemma \ref{lem:basic estimate}
  and we set $\bar\tau_n:=\bar \tau(\eps_n)$,
  $\bar T_n:=T(\eps_n)\ge T_n$.

  If \eqref{eq:89} holds, we set 
  $\bar\tau_n:=\bar\tau(\eps_n, T_n)>0$,
  $\bar T_n:=\bar T(T_n)\ge T_n$
  provided by Lemma
  \ref{lem:basic estimate2}.
  
  We can find a decreasing sequence $\sigma_n\downarrow0$ satisfying 
  $\sigma_n\le \min_{1\le m\le n}\bar \tau_m$ 
  and
  a family $\phi_{\tau}$ by choosing
  \begin{displaymath}
    \phi_{\tau}:=\varphi_{\eps_n,\cU(\tau, \bar T_n)}
    \quad\text{whenever
    }\sigma_{n+1}<\tau\le \sigma_n.
  \end{displaymath}
  By construction
  \begin{displaymath}
    \Lip[\phi_{\tau}-\phi]\le \eps_n
    \quad\text{if }\sigma_{n+1}<\tau\le \sigma_n,
  \end{displaymath}
  so that $\lim_{\tau\downarrow0}    \Lip[\phi_{\tau}-\phi]=0$. 
  
  We first consider the case \textbf{$T_\star(u) < +\infty$}. If $T_\star(u) < +\infty$, then the range $\range u$ is compact and \eqref{eq:87} holds. Lemma \ref{lem:basic estimate} shows that $\DMM\tau{\phi_\tau}{u(0)}$ is not empty for $\tau\in(0, \sigma_1)$. Moreover, if $U_\tau\in\DMM\tau{\phi_\tau}{u(0)}$ is any selection depending on $\tau\in(0, \sigma_1)$, we have $U_\tau([0, +\infty))\subset \range u$ and $\phi(U_\tau(t)) \le \phi(u(t\land T_n))$ for every $t\ge 0$, $\sigma_{n+1} < \tau \le \sigma_n$. By Lemma \ref{le:well-known}(iv), every decreasing vanishing
  sequence $k\mapsto\tau(k)$ 
  admits a further subsequence (still denoted by $\tau(k)$)
  such that $U_{\tau(k)}$ converges in the topology of compact convergence  
  to a limit $v\in \solution\phi$. It holds that $\phi(v(t)) \le \phi(u(t))$ for all $t\ge 0$, which implies $u(t) = v(t)$ for all $t\ge 0$ by Theorem \ref{thm: max gf}(4) since $u$ is minimal and $\range v = \range u$. As the limit is unique, we obtain
  \begin{displaymath}
  \lim_{\tau\downarrow0}\dsfD \infty(u,
  \DMM\tau{\phi_{\tau}}{u(0)})=0,
  \end{displaymath}
 showing that $u$ is strongly approximable according to Definition \ref{def:AGF}. For every $T > 0$, $\tau\in(0, \sigma_1)$, it holds that $\DMM\tau{\phi_{\tau}}{u(0), T} = \{U|_{[0, T]} \ | \ U\in\DMM\tau{\phi_{\tau}}{u(0)}\}$; hence, by Remark \ref{rem:approximability}, $u$ is also strongly approximable in every compact interval.
  
 Now, we consider the case $T_\star(u) = +\infty$. 
  Let us fix $T>0$ and take $\bar n=\min\{n\in \N:T_n\ge T+1\}$.
  Lemma \ref{lem:basic estimate} and
  \ref{lem:basic estimate2} show
  that $\DMM\tau{\phi_\tau}{u(0),T+1}$ 
  is not empty whenever $\tau\le \sigma_{\bar n}$.
  Moreover, if $U_\tau\in \DMM\tau{\phi_\tau}{u(0),T+1}$
  is any selection depending on $\tau\in(0, \sigma_{\bar n})$,  
  we have $\phi(U_\tau(t))\le \phi(u(t))$ for every $t\in [0,T+1]$. According to Lemma \ref{le:apriori}(ii) and to (\ref{eq:104}) and (\ref{eq:105}), there exist $\tau_\star\in(0, \sigma_{\bar n})$ and a constant $C > 0$ independent of $\tau$ such that
  \begin{equation}\label{eq:2.28+2.29}
  |U_\tau(t) - U_\tau(s)| \le 2C\sqrt{\tau} + C|t-s|^{1/2} \quad \text{for every } s,t \in [0, T+1], \ \tau\in(0, \tau_\star). 
  \end{equation}   
  We define $S_\tau:= \inf\{t\in[0, T+1] \ | \ \phi(U_\tau(t)) \le \phi(u(T+1))\}$ for $\tau\in (0, \sigma_{\bar n})$ and $\tilde{S}:= \liminf_{\tau\downarrow 0} S_\tau$. The varying times $S_\tau$ serve as auxiliary final times in order to prove convergence of $U_\tau$. We set $\gamma_\tau:= (T+1 - S_\tau) \land \tau$. As the piecewise constant functions $U_\tau$ are left-continuous by definition and $\phi(U_\tau(T+1)) \le \phi(u(T+1))$, it holds that $S_\tau < T+1$ (thus $\gamma_\tau > 0$) 
  and $\phi(U_\tau(S_\tau + \gamma_\tau)) \le \phi(u(T+1))$. The plan is as follows. We show that $\tilde{S} > 0$, we prove that $U_\tau$ converges to $u$ uniformly in $[0, S]$ for every $0 < S < \tilde{S}$, and we conclude by proving that $\tilde{S} = T+1$. 
  
   There exists a vanishing sequence $l\mapsto\tau(l)$ such that $\lim_{l\uparrow \infty}S_{\tau(l)} = \tilde{S}$. A contradiction argument shows that $\tilde{S} > 0$. Suppose that $\tilde{S} = 0$; then $U_{\tau(l)}(S_{\tau(l)} + \gamma_{\tau(l)})$ converges to $u(0)$ by \eqref{eq:2.28+2.29} and $\phi(u(0)) = \lim_{l\uparrow\infty}\phi(U_{\tau(l)}(S_{\tau(l)} + \gamma_{\tau(l)})) \le \phi(u(T+1))$ in contradiction to $\phi(u(T+1)) < \phi(u(0))$ by the minimality of $u$ and Theorem \ref{thm: max gf}(5). Hence, $\tilde{S} > 0$. For every $0 < S < \tilde{S}$ and sufficiently small $\tau$, it holds that $U_\tau([0, S]) \subset u([0, T+1])$ so that  
  by Lemma \ref{le:well-known}(iii), every decreasing vanishing
  sequence $k\mapsto\tau(k)$ 
  admits a further subsequence (still denoted by $\tau(k)$)
  such that 
  $U_{\tau(k)}$ converges uniformly in $[0, S]$ 
  to a limit $v\in \rmC^1([0, S], \H)$ solving (\ref{eq:1}) in $[0, S]$. 
Moreover, since we have $v([0, S]) \subset \range u$ 
and $\phi(v(t)) \le \phi(u(t))$ for all $t\in[0, S]$,
we deduce that
$u(t) = v(t)$ for all 
$t\in[0, S]$ by the \minimality\ of $u$, Remark \ref{rem:TGF} and Theorem \ref{thm: max gf}(4).
Since the limit is unique, 
we can now infer that $\lim_{\tau \downarrow 0}\dsfd S (U_\tau, u) = 0$ for every $S < \tilde{S}$. Using \eqref{eq:2.28+2.29}, we obtain
\begin{eqnarray*}
\limsup_{l\uparrow\infty}{|U_{\tau(l)}(S_{\tau(l)}) - u(\tilde{S})|} &\le& \limsup_{l\uparrow\infty}{\left(|U_{\tau(l)}(S_{\tau(l)}) - U_{\tau(l)}(S)| + |U_{\tau(l)}(S) - u(\tilde{S})|\right)} \\
&\le& \limsup_{l\uparrow\infty}{\left(2C\sqrt{\tau(l)} + C|S_{\tau(l)} - S|^{1/2} + |U_{\tau(l)}(S) - u(\tilde{S})|\right)} \\
&\le& C|\tilde{S} - S|^{1/2} + |u(S) - u(\tilde{S})|
\end{eqnarray*}
for every $S < \tilde{S}$ and therefore $u(\tilde{S}) = \lim_{l\uparrow \infty} U_{\tau(l)}(S_{\tau(l)}) = \lim_{l\uparrow \infty} U_{\tau(l)}(S_{\tau(l)} +\gamma_{\tau(l)})$. It follows that $\phi(u(\tilde{S}))  = \lim_{l\uparrow\infty}\phi(U_{\tau(l)}(S_{\tau(l)} + \gamma_{\tau(l)})) \le \phi(u(T+1)))$ which implies $u(\tilde{S}) = u(T+1)$ as $\tilde{S} \le T+1$. Since the minimal solution $u$ is injective for $T_\star(u) = +\infty$ by Theorem \ref{thm: max gf}(5), it follows that $\tilde{S} = T+1 = \lim_{\tau\downarrow 0} S_\tau$. So we obtain 
 \begin{equation}\label{eq: sonne}
  \lim_{\tau\downarrow0}\dsfD T(u|_{[0,T]},
  \DMM\tau{\phi_{\tau}}{u(0),T})=0
  \end{equation}
by the preceding argument and the fact that $\DMM\tau{\phi_\tau}{u(0),T} = \{U|_{[0, T]} \ | \ U\in\DMM\tau{\phi_\tau}{u(0),T+1}\}$ for $\tau\in(0, \sigma_{\bar n})$. 
This shows that $u$ is strongly approximable in every compact interval. 
  
    
If $\H$ has finite dimension, then 
Remark \ref{rem:approximability} 
shows that 
$u$ is also strongly approximable. 
  
If \eqref{eq:87} holds, 
then Lemma \ref{lem:basic estimate} shows
that $\DMM\tau{\phi_\tau}{u(0)}$ 
is not empty for $\tau\in(0, \sigma_1)$; 
hence, according to Remark \ref{rem:approximability}, $u$ is also strongly approximable. The same can be shown if (\ref{eq: liminf is zero}) holds, see Remark \ref{rem:liminf}.  
\end{proof}

The next step in the proof of De Giorgi's conjecture is to show that
we can approximate any gradient flow curve by a sequence of \minimal\ 
gradient flows for slightly (in the Lipschitz norm) modified energies,
and then to combine that convergence result and Theorem
\ref{thm:appmax} by Lemma \ref{le:appsol}. 
This will be first considered in the one-dimensional setting.

\section{The one dimensional setting}\label{sec: one dimensional setting}

In this section we want to study the one-dimensional case $\H=\R$. Just for this section, 
we will call $E:=-\phi$ and
we consider a 
continuously
differentiable function 
$E: \R \rightarrow \R$ with derivative $f:=E'$.
\begin{proposition}\label{prop: one dimensional setting}
Let 
$u\in \solution{-E}$ be an eventually \minimal\ solution (see Definition
\ref{def:emaxsolution}), i.e. 
\begin{equation}
  \label{eq:technical}
  \text{there exists
    $T>0$ with
    $u'(T)\neq 0$ and $t\mapsto u(t+T)$ is \minimal.}
\end{equation}
\normalcolor Then 
there exist a sequence of energies $E_\eps\in \rmC^1(\R)$ and a sequence of curves
$u_\eps\in \minsolution{-E_\eps}$ with $u_\eps(0)=u(0)$ such that 
\begin{equation}
  \label{eq:59}
   E_\eps'=E'\quad\text{in }\R\setminus u([0,T]),\quad
  \lim_{\eps\downarrow0}\Lip[E_\eps-E]=0,\quad
  \lim_{\eps\downarrow0}\dsfd\infty(u,u_\eps)=0.
\end{equation}
\end{proposition}
\begin{proof}
In order to simplify the notation, we may assume w.l.o.g. that
$u(0)=0$. We notice that $u$ is a monotone function. This can be shown by contradiction: suppose that $u$ is not monotone and choose $a, b \in (0, \infty)$ with $u'(a) > 0$ and $u'(b) < 0$, w.l.o.g. $a < b$. Then there exists $\gamma > 0$ such that $u$ is strictly increasing on $[a, a+\gamma]$ and strictly decreasing on $[b-\gamma, b]$. It holds that $u(a) < u(b)$: otherwise there would be $s\in (a+\gamma, b]$ with $u(s) = u(a)$ which, by (\ref{eq:4}), would imply $u(t) = u(a)$ for all $t\in [a, s]$ contradicting the strict monotonicity of $u$ in $[a, a+\gamma]$. A similar argument yields $u(b) < u(a)$, a contradiction. Hence, $u$ is either increasing or decreasing. If Proposition \ref{prop: one dimensional setting} holds for increasing solutions, then, by obvious reflection arguments, it also holds for decreasing solutions, and thus for all solutions $u$.     
So we may assume that $u$ is an increasing function 
whose range $\range u$ is an interval of the form $[0,R)$, $R\in
(0,\infty]$ or $[0,R]$, $R\in (0,\infty)$. 
We define the left-continuous pseudo-inverse map
$\sft:[0,R)\to [0,T_\star(u))$ 
\begin{equation}
\label{eq:69}
\sft(x) := \min\{t\ge0: \ u(t) = x\},\quad
\text{satisfying}\quad
u(\sft(x))=x\quad\text{for every }x\in [0,R).
\end{equation}
The map $\sft$ is an increasing function, in particular it is a
function of bounded variation in any compact interval of $[0,R)$; 
\eqref{eq:69} yields that the set $D$ of points in $[0,R)$ where $\sft$ is
differentiable
coincides with the set $\{x\in [0,R):E'(x)>0\}$ and
$u'(\sft(x))\sft'(x)=1$ for every $x\in D$.
Lebesgue differentiation theorem shows that $D$ 
has full measure in $[0,R)$.
Since $u'=E'(u)$ we deduce 
that 
$$\sft'(x) = \frac{1}{u'(\sft(x))} = \frac{1}{E'(x)} =
\frac{1}{f(x)}\quad\text{for every }x\in D,$$
and the property
\begin{equation}
  \label{eq:70}
  \int_{[0,x]\cap D} \frac{1}{f(y)}\,\mathrm d y\le
  \sft(x)<\infty\quad\text{for every }x<R. 
\end{equation}
\eqref{eq:technical}
yields
\begin{equation}
  \mathscr L^1\big(\{t\in (T,T_\star(u)): u(t)\not\in
  D\}\big)=0,
\label{eq:44}
\end{equation}
so that $\sft$ is locally absolutely continuous in the
interval $[u(T),R)$.
Since the distributional derivative of $\sft$ is a Radon measure on
$[0,R)$,
there exists a nonnegative finite Borel measure $\mu$ supported on $[0,u(T)]$ such that
\begin{displaymath}
  \sft(x)=\int_{[0,x]\cap D} \frac1{f(r)}\,\mathrm d
  r+\mu([0,x))\quad
  \text{for every }x\in [0,R).
\end{displaymath}
Notice that $\mu([0,R))\le T$.
We can approximate $\mu$ by convolution 
(we will still denote by $\mu$ its trivial extension to $0$ outside the
interval $[0,R)$)
$$m_\eps(x):=\mu\ast
\kappa_\eps(x)=\frac 1\eps\int_{-\infty}^{+\infty} \kappa((x-y)/\eps)\,\rmd\mu(y)$$
where $\kappa$ is a shifted standard $\mathrm C^\infty_{\rm c}$ 
mollifier (see e.g. [\cite{AmbrosioFuscoPallara00}, p.41]) with support in $[0,1]$ and we define
\begin{displaymath}
  \sft_\eps(x):=\int_{0}^x \frac 1{f(r)}+m_\eps(r) \,\mathrm d r = \int_{0}^x \frac {1+m_\eps f}{f}(r) \,\mathrm d r  
\end{displaymath} 
for $x\in [0, R)$. 

We denote by $\mathrm J_\sft\subset[0, u(T)]$ the at most countable 
set of discontinuity points of $\sft$, which coincides with the set of
atomic points of $\mu$ (i.e.~$\{x\in [0,u(T)]:\mu\{x\}>0\}$).
Since $m_\eps\mathscr L^1$ converge to $\mu$ as $\eps\downarrow0$ in
the weak topology of finite positive measures, we obtain 
\begin{equation*}
\lim_{\eps\to 0}\int_{0}^{x}{m_\eps(r) \ \mathrm d r} = \mu([0, x))
\quad
\text{for every }x\in [0,R)\setminus \mathrm J_\sft,
\end{equation*}
see e.g. Proposition 1.62(b) and Theorem 2.2 in
\cite{AmbrosioFuscoPallara00}. 
We used the fact that the support of $m_\eps$ is contained in $[0, u(T)+ \eps]$. The convergence 
\begin{equation*}
\sft_\eps(x) \to \sft(x) \text{ as } \epsilon\to0  
\end{equation*}
for all $x\in[0, R)\setminus J_\sft$ directly follows. 
{\normalcolor Moreover, there exists $\bar{\eps} > 0$ such that $(u(T)-\bar{\eps}, u(T)] \subset D$; hence for $\eps\in(0, \bar{\eps})$, the support of $m_\eps$ is contained in $[0, u(T)]$ and
\begin{equation}
  \label{eq:115}
  \text{$\sft_\eps(x)=\sft(x)$ for every $x\ge u(T)$,}
\end{equation}
since $\int_0^{u(T)}m_\eps(r)\,\rmd r=\mu([0,u(T)])$.}
Let us now 
consider the map $\sft_\eps$ {\normalcolor for $\eps\in (0, \bar{\eps})$} fixed. It is locally absolutely continuous, strictly increasing and differentiable for $\mathcal{L}^1$-a.e. $x\geq 0$ with 
\begin{equation*}
\sft_\eps'(x) = \frac{1+m_\eps(x)f(x)}{f(x)} > 0,\quad
\lim_{x\uparrow R}\sft_\eps(x)=T_\star(u)=:T_\star.
\end{equation*}
Thus  the inverse map $u_\eps: [0, T_\star) \rightarrow [0,R)$ is
locally absolutely continuous with 
\begin{equation}
  \label{eq:116}
u_\eps'(t) = 
\frac{f(u_\eps(t))}{1+m_\eps(u_\eps(t))f(u_\eps(t))} \text{ for }
\mathcal{L}^1\text{-a.e. } t,
\quad
\normalcolor u_\eps(t)=u(t)\quad\text{for }t\ge T.
\end{equation}
Moreover, if $T_\star<\infty$, we see that $\lim_{t\uparrow
  T_\star}u_\eps(t)
=R=u(T_\star)$ and we can extend $u_\eps$ to the whole real line
by setting $u_\eps(t)=u(T_\star)$ for $t\ge T_\star$.

So we obtain that $u_\eps$ satisfies $u_\eps'(t)=E_\eps'(u_\eps(t))$ for all $t\in [0, \infty)$, for the energy $E_\eps: \mathbb{R} \rightarrow \mathbb{R}$ with
\begin{displaymath}
  E_\eps'=\frac f{1+m_\eps f}
\end{displaymath} 
(and initial value $u_\eps(0) = 0$). Moreover, since $\sft_\eps$ is absolutely continuous, the set 
\begin{displaymath}
\{t\in [0, T_\star): \ E_\eps'(u_\eps(t)) = 0\} \subset
\sft_\eps([0,R)\setminus D)
\end{displaymath} 
has Lebesgue measure $0$.

Since $E_\eps'$ is uniformly bounded in every bounded interval, 
the family $u_\eps$ is uniformly
Lipschitz in every bounded 
interval by \eqref{eq:116}; in order to prove that it converges to $u$ as
$\eps\downarrow0$ it is sufficient to characterize its limit $\tilde
u$ along a convergent subsequence $k\mapsto u_{\eps(k)}$,
$\eps(k)\downarrow0$ {\normalcolor(which exists by Ascoli-Arzel\`a theorem).} 


Since for all $x\in [0, R)\setminus \mathrm J_\sft$ we have
\begin{equation*}
u(\sft(x))=x = u_{\eps(k)}(\sft_{\eps(k)}(x)) \to \tilde{u}(\sft(x))
\quad\text{as }k\uparrow\infty;
\end{equation*}
since $\sft$ is left-continuous, we get 
\begin{equation*}
\tilde{u}(t) = u(t) \text{ for all } t\in \sft([0,R)).
\end{equation*}
Since $u$ is continuous and locally constant in the interior of 
$[0,T_\star)\setminus \sft([0,R))$ and $\tilde u$ is monotone, we
conclude that $u\equiv \tilde u$ on $[0,\infty)$.
%
Hence, $u_\eps$ is {\normalcolor converging uniformly} to $u$. 
 
In the last part of the proof, we show that $E_\eps'=:f_\eps$ is
converging uniformly to $E'=f$ 
on $\mathbb{R}$. We notice that the support of $\mu$ is a compact set
included in $\tilde N:=
[0,u(T)]\setminus D$, where $f$ vanishes. Hence, the support of $m_\eps$
is contained in the $\eps$-neighborhood {\normalcolor$\tilde N_\eps:=\{x\in [0, u(T)]:
\operatorname{dist}(x,\tilde N)\le \eps\}$ of $\tilde N$ for $\eps\in(0, \bar{\eps})$} so that 
$f=f_\eps$ in the complement of $\tilde N_\eps$. On the other hand,
since $0\le f_\eps\le f$ {\normalcolor on $[0, R)$}, we get
\begin{displaymath}
  \sup_{x\in \R}|f(x)-f_\eps(x)|\le 2\sup_{x\in \tilde N_\eps} f(x)\downarrow
  0\quad\text{as }\eps\downarrow 0,
\end{displaymath}
since $f$ is uniformly continuous in every compact subset of $\R$ and 
$f\equiv 0$ on $\tilde N$.
%
%
%
%
\end{proof}

By applying Lemma \ref{le:appsol}, Remark \ref{rem:emaxsolution} and Theorem \ref{thm:appmax}, we can now 
easily prove that in the one dimensional case any solution of \eqref{eq:1} is strongly
approximable.
In the next section, we will use Proposition 
\ref{prop:  one dimensional setting} as an inspiring guide to study the problem in an
arbitrary finite dimensional setting.

\section{Strongly approximable solutions}\label{sec: conclusion}

We consider an arbitrary non-constant solution $u\in
\solution\phi$ {\normalcolor for $\phi\in\rmC^1(\H)$}. 
Let $v\in \minsolution\phi$ be the unique \minimal\ solution 
with $u\succ v$ (see Theorem \ref{thm: max gf}) and 
$\range u\subset \range v\subset \overline{\range u}$, and set $T_\star:=T_\star(v)$.


We know that there exists an increasing $1$-Lipschitz map 
$\sfz: [0, +\infty)\to [0, +\infty)$ such that $u(t)=v(\sfz(t))$.
We also know that the restriction of $v$ to $[0,T_\star)$ is an
homeomorphism with  
\begin{equation}
R:= 
\begin{cases}
\range v &\text{ if } T_\star = +\infty,\\
\range v \setminus \{v(T_\star)\} &\text{ if } T_\star < +\infty 
\end{cases}
\end{equation}
whose inverse will be denoted by $\widetilde\sft:R\to
[0,T_\star)$. We will set
\begin{equation}
  \label{eq:54}
  \sfx(t):=\int_0^t |v'(s)|\,\rmd s,\quad
 \normalcolor L_\star:=\lim_{t\uparrow T_\star}\sfx(t)=
  \int_0^{+\infty}|v'(s)|\,\rmd s;
\end{equation}
notice that $\sfx\in \rmC^1([0,T_\star))$ with $\sfx'(t)>0$ a.e.~so
that it admits a locally absolutely continuous inverse that we will denote by
$\sft:[0,L_\star)\to [0,T_\star)$. 
 
The arc-length parametrization of $R$ is then given by 
\begin{equation}
\label{eq:61}
\widetilde\sfx: R \rightarrow [0, L_\star),\quad
\widetilde\sfx(y) := \sfx(\widetilde\sft(y))=\int_0^{\widetilde\sft(y)}{|v'(s)| \ \mathrm d
  s},\quad
y\in R.
\end{equation}  
Notice that $\tilde\sfx$ is an homeomorphism between $R$ and $[0,L_\star)$
which associates to every point $u(t)\in R$ the length of the curve
$u([0,t])$; in particular $\widetilde \sft(y)=\sft(\widetilde \sfx(y))$.

Its inverse $\sfy:=(\widetilde\sfx)^{-1}:[0,L_\star)\to R$
is the arc-length parametrization of the curve $v$, defined by
\begin{equation}
  \label{eq:64}
  \sfy(x)=v(\sft(x)).
\end{equation}

We can now consider the one-dimensional energy obtained by rectifying 
the graph of $v$
\begin{equation}
E: [0, L_\star) \rightarrow \R,\quad
E:=-\phi\circ \sfy,\quad
E(\sfx(t))=-\phi(v(t)),\label{eq:62}
\end{equation}
which is 
continuously differentiable with derivative
\begin{equation}
  \label{eq:66}
  E'(x)= - \langle \nabla\phi(\sfy(x)),v'(\sft(x))\rangle \sft'(x)=
  |\nabla\phi(\sfy(x))|^2\frac{1}{|v'(\sft(x))|}=
  |\nabla\phi(\sfy(x))|. 
\end{equation}
{\normalcolor If $T_\star < +\infty$, then $L_\star < +\infty$ and $E$ is Lipschitz, so there is a continuous extension of the energy to $[0, L_\star]$ which we still denote by $E$.}  
 
The next lemma shows that $u$ gives rise to a solution of
$\solution{-E}$ 
via a suitable rescaling.
\begin{lemma}
  The curve
\begin{equation}
  \label{eq:63}
  \sfu: [0, \infty) \rightarrow \R,\quad
  \sfu(t):=
  \begin{cases}
  \widetilde\sfx(u(t)) &\text{ for } t < T_\star(u),\\
  L_\star &\text { for } t\ge T_\star(u) \text{ if } T_\star(u) < +\infty
  \end{cases}
\end{equation}
belongs to $\solution{-E}$, i.e.
\begin{equation}
  \label{eq:67}
\sfu'(t) = E'(\sfu(t))\quad \text{for all } t\ge 0. 
\end{equation}
Moreover, if $u$ is eventually \minimal, then $\sfu$ is also eventually \minimal.
\end{lemma} 
\begin{proof}


 In order to check \eqref{eq:67} we first observe that $\sfu(t)=\sfx(\widetilde\sft(u(t)))=\sfx(\widetilde\sft(v(\sfz(t))))=
  \sfx(\sfz(t))$ for $t\in [0, T_\star(u))$ and  
\begin{displaymath}
   \sfx'(t)=|v'(t)|=|\nabla\phi(v(t))|=
   |\nabla\phi(\sfy(\sfx(t))|=
   E'(\sfx(t)) \text{ for } t\in[0, T_\star)
 \end{displaymath}
 since $v(t)=\sfy(\sfx(t))$. Therefore
 \begin{displaymath}
   \sfu'(t)=\sfx'(\sfz(t))\sfz'(t)=E'(\sfx(\sfz(t)))\sfz'(t)=
   E'(\sfu(t))\sfz'(t) 
 \end{displaymath}
 for a.e. $t\in[0, T_\star(u))$. On the other hand, we know that $\sfz'(t)=1$ whenever
 $|\nabla\phi(u(t))|\neq0$,
 i.e.~if $|\nabla\phi(\sfy(\sfu(t))|=E'(\sfu(t))\neq 0$. If $T_\star(u) < +\infty$, the limit $L_\star = \lim_{t\uparrow T_\star(u)} \sfx(\sfz(t))$ is finite and we can extend $E$ to $[0, L_\star]$ with $E'(L_\star) = \lim_{x\uparrow L_\star} E'(x) = 0$. Our calculations show that $\sfu\in\solution{-E}$.  

Finally, let us assume that 
$u$ is eventually \minimal; this is equivalent to say that for some
$T<T_\star(u)$ with $u'(T)\neq 0$ we have $\sfz'(t)\equiv 1$ in $(T,T_\star(u))$, 
so that $\sfu$ is also
eventually \minimal. 
\end{proof}
%
%
Let us assume that $u$ is eventually \minimal, according to Definition \ref{def:emaxsolution}; in particular $\sfu$ then satisfies (\ref{eq:technical}).
Arguing as in Proposition \ref{prop: one dimensional setting}, 
we associate to $E$ and $\sfu$ energies $E_\eps: [0, L_\star) \rightarrow
\R$ ({\normalcolor with continuous extension to $[0, L_\star]$ if $T_\star \lor L_\star < +\infty$}) and curves $\sfu_\eps: [0, \infty) \rightarrow \R$ satisfying
\begin{equation}\label{eq: derivative of Eeps}
  E_\eps'(x) = \frac{|\nabla\phi(\sfy(x))|}{1 + m_\eps(x)|\nabla\phi(\sfy(x))|} 
\end{equation} 
with $m_\eps$ chosen as in the proof of Proposition \ref{prop: one
  dimensional setting}, and 
\begin{equation}\label{eq: gf eq for Ueps}
\sfu_\eps'(t) = E_\eps'(\sfu_\eps(t)) \text{ for all }t\in [0, +\infty).
\end{equation}
The set
\begin{displaymath}
\{t\in[0, T_\star(u)): \ E_\eps'(\sfu_\eps(t)) = 0\}
\end{displaymath}
has Lebesgue measure $0$, and $\sfu_\eps$ is converging locally uniformly to
$\sfu$ as $\eps\to0$.
We observe that $E_\eps$ satisfies up to an additive constant 
\begin{equation}
  \label{eq:74}
  E_\eps(x)=
  \int_0^x
  \frac{|\nabla\phi(\sfy(r))|}{1+m_\eps(r)|\nabla\phi(\sfy(r))|}\,\rmd
  r
  =
    \int_0^{\sft(x)}
  \frac{|\nabla\phi(v(t))|}{1+m_\eps(\sfx(t))|\nabla\phi(v(t))|}|v'(t)|\,\rmd
  t.
\end{equation}
Now, we translate this one dimensional setting with the approximation 
by $E_\eps$ and $\sfu_\eps$ back to the initial situation with $\phi$ and $u$.

\begin{lemma}\label{lem: ueps is max gf for phieps}
Let us suppose that $u$ is eventually \minimal\ and that there exist 
$\phi_\eps\in \rmC^1(\H)$ satisfying 
\begin{equation}\label{eq: phieps on SigmaT}
  \phi_\eps(y) = -E_\eps(\tilde\sfx(y))
  =
  -\int_0^{\tilde\sft(y)}
  \frac{|\nabla\phi(v(t))|}{1+m_\eps(\sfx(t))|\nabla\phi(v(t))|}|v'(t)|\,\rmd
  t\quad\text{on } R,
\end{equation} 
and 
\begin{equation}\label{eq: gradient of phieps on SigmaT}
\nabla\phi_\eps(y) = \frac{\nabla\phi(y)}{1 + m_\eps(\widetilde
  \sfx(y)) |\nabla\phi(y)|} 
\text{ for all } y\in R. 
\end{equation}
Then the curve
\begin{equation}
 u_\eps: [0, +\infty)\rightarrow \H, \quad u_\eps(t):=
 \begin{cases}
 \sfy(\sfu_\eps(t)) &\text{ for } t < T_\star(u), \\
 {\normalcolor u(T_\star(u))} &\text{ for } t\ge T_\star(u) \text{ if } T_\star(u) < +\infty 
 \end{cases}
\end{equation}
is a \minimal\ gradient flow for $\phi_\eps$. Moreover, $u_\eps$ is converging locally uniformly to $u$ as $\eps\to0$. 
\end{lemma}

\begin{proof}
We just observe that for a.e.~$x\in [0,L_\star)$
\begin{equation}
  \label{eq:71}
  \sfy'(x)=v'(\sft(x))\sft'(x)=-\nabla\phi(v(\sft(x)))\frac{1}{|v'(\sft(x))|}=
  -\frac{\nabla\phi(\sfy(x))}{|\nabla\phi(\sfy(x))|}
\end{equation}
so that 
\begin{align*}
  u_\eps'(t)&=\sfy'(\sfu_\eps(t))\sfu_\eps'(t)=
              -\frac{\nabla\phi(u_\eps(t))}{|\nabla\phi(u_\eps(t))|} 
              \frac{|\nabla\phi(u_\eps(t))|}{1+m_\eps(\sfu_\eps(t))|\nabla\phi(u_\eps(t))|}
              =-\nabla\phi_\eps(u_\eps(t)).
\end{align*}
The convergence of $u_\eps$ is a consequence of the convergence of $\sfu_\eps$.
\end{proof}

It remains to show that there indeed exist energies $\phi_\eps: \H\to
\R$ satisfying the assumptions of Lemma \ref{lem: ueps is max gf for
  phieps} and converging to $\phi$ in the Lipschitz
seminorm. 
Note that (\ref{eq: phieps on SigmaT}) which is not used in the proof of Lemma \ref{lem: ueps is max gf for phieps} should give an idea of how to construct $\phi_\eps$.

\begin{lemma}\label{lem: whitney extension}
Let us suppose that $\H$ has finite dimension and $u$ is an eventually
\minimal\ solution to \eqref{eq:1}. 
There exist {\normalcolor continuously differentiable 
functions} $\phi_\eps: \H \rightarrow \R$ such that (\ref{eq: phieps
  on SigmaT}) (up to an additive constant) and (\ref{eq: gradient of phieps on SigmaT}) is
satisfied and 
\begin{equation}
  \label{eq:75}
  \lim_{\eps\downarrow0}\sup_{\H}|\phi_\eps-\phi|+\Lip[\phi_\eps-\phi]=0.
\end{equation}
\end{lemma} 
\begin{proof}
Let us fix a time $T<T_\star(u)$ such that $m_\eps(\widetilde \sfx(y)) = 0$ for all $y\in R\setminus u([0, T])$ and $\eps > 0$ and choose $T_1\in (T,T_\star(u))$ so that 
$\phi(u(T))>\phi(u(T_1))>\inf_{\range
  u}\phi$. We consider the compact sets $K:=u([0,T])$ and
$K_1:=u([0,T_1])$ and the open set $A:=\{w\in \H:\phi(w)>\phi(u(T_1))\}$
which contains $K$. We can 
find a smooth function $\psi:\H\to [0,1]$ such that 
\begin{equation}
  \label{eq:72}
  \psi(w)\equiv 1\quad\text{on }K,\quad 
  \psi(w)=0\quad \text{on }\H\setminus A,
\end{equation}
and 
\begin{equation}\label{eq: 72 ii} 
\nabla\psi \equiv 0 \quad\text{on } K, \quad \sup_\H |\nabla\psi| < +\infty. 
\end{equation}
The construction of $\psi$ is standard: 
there exists $\delta > 0$ such that the distance function $\operatorname{dist}(x, K)$ satisfies
\begin{displaymath}
\operatorname{dist}(x, K) \le 4\delta \quad \Rightarrow \quad x\in A. 
\end{displaymath}
The composition of $\operatorname{dist}(x, K)$ with $\eta(d):=\frac{1}{\delta}(\delta-(d-2\delta)_+)_+$ then yields a $\frac{1}{\delta}$-Lipschitz function taking value $1$ in a neighborhood of radius $2\delta$ around $K$ and vanishing if $\operatorname{dist}(x, K)\ge 3\delta$. Taking the convolution of $\eta\circ \operatorname{dist}(\cdot, K)$ with a smooth kernel with support in $\{|x|\le \delta\}$, we obtain a suitable function $\psi$. 
 
Let us define $\delta_\eps: K_1 \rightarrow \R, \
\delta_\eps(w):= 
-E_\eps(\widetilde\sfx(w)) - \phi(w)$. Applying Whitney's Extension Theorem [see
e.g. \cite{Hoermander98}, Theorem 2.3.6], we aim to extend
$\delta_\eps$ to a $\rmC^1$ function in $\H$ 
with gradient $Q_\eps: \H\to \H$ satisfying
$Q_\eps(w) = F_\eps(w) - F(w)$ on $K_1$, in which
\begin{displaymath}
F_\eps(w) := \frac{\nabla\phi(w)}{1 + m_\eps(\tilde \sfx(w)) |\nabla\phi(w)|}, \quad F(w):= \nabla\phi(w).
\end{displaymath} 
For that purpose, since $\delta_\eps$ and $Q_\eps$ are continuous and
$\phi\in \rmC^1(\H)$, we only need to check 
if for $w_n, {\bar w}_n\in K_1$ with $w_n\neq \bar {w}_n, \
\lim_{n\to0} |\bar {w}_n - w_n| = 0$, 
it holds that 
\begin{equation}\label{eq: deltaeps is C1}
  \lim_{n\to\infty}\frac{-E_\eps(\widetilde\sfx(\bar {w}_n)) + E_\eps(\widetilde \sfx(w_n)) - \langle F_\eps(w_n), \bar {w}_n - w_n \rangle}{|\bar {w}_n - w_n|}= 0. 
\end{equation}
Up to extracting a subsequence, it is not restrictive to assume that
$\bar  w_n$ and $w_n$ converge to a common limit point $w$. 
By using the \minimal\ flow $v$ we can also find points
$t_n=\widetilde\sft(w_n),\bar  t_n=
\widetilde\sft(\bar w_n)$
converging to some $t$
such that $\bar  w_n=v(\bar t_n)$, $w_n=v(t_n)$, $w=v(t)$.
Notice that 
\begin{align*}
  E_\eps(\tilde\sfx(w_n))-
  E_\eps(\tilde\sfx(\bar w_n))
  &=
    \int_{\bar t_n}^{t_n}
   \frac{|\nabla\phi(v(r))|}{1+m_\eps(\sfx(r))|\nabla\phi(v(r))|}|v'(r)|\,\rmd
    r\\
  \langle F_\eps(w_n), \bar {w}_n - w_n \rangle
  &=
    \frac{\langle \nabla\phi(v(t_n)),v(\bar t_n)- v(t_n)\rangle}
    {1+m_\eps(\sfx(t_n))|\nabla\phi(v(t_n))|}
\end{align*}
%
If $\nabla\phi(w) = 0$, then (\ref{eq: deltaeps is C1}) directly follows from the fact that 
\begin{equation*}
  \big|E_\eps(\tilde\sfx(w_n))-
  E_\eps(\tilde\sfx(\bar w_n))\big|\le \Big|\int_{\bar t_n}^{t_n}
  |\nabla\phi(v(r))||v'(r)|\,\rmd
r\Big|
= |\phi(\bar w_n) - \phi(w_n)|,
\end{equation*}  
so that 
\begin{displaymath}
  \limsup_{n\to\infty}\frac{\big|E_\eps(\tilde\sfx(w_n))-
    E_\eps(\tilde\sfx(\bar w_n))\big|}{|w_n-\bar w_n|}\le 
  \limsup_{n\to\infty}\frac{|\phi(\bar w_n) - \phi(w_n)|}{|w_n-\bar w_n|}=0,
\end{displaymath}
and
\begin{displaymath}
  \limsup_{n\to\infty}\frac{\big|\langle F_\eps(w_n), \bar {w}_n - w_n
    \rangle|}{|\bar {w}_n - w_n|}
  \le \limsup_{n\to\infty} |\nabla\phi(w_n)|=0.
\end{displaymath}
If $|\nabla\phi(w)| \neq 0$, then 
\begin{displaymath}
  \lim_{n\to\infty}\frac{v(t_n)-v(\bar t_n)}{t_n-\bar
    t_n}=v'(t)=-\nabla\phi(v(t))=
  -\nabla\phi(w)\neq 0,
\end{displaymath}
and
\begin{align*}
  \lim_{n\to\infty}&\frac{-E_\eps(\widetilde\sfx(\bar {w}_n)) +
  E_\eps(\widetilde \sfx(w_n)) - \langle F_\eps(w_n), \bar {w}_n - w_n
  \rangle}{t_n-\bar t_n}
                     \\&=
  \frac{|\nabla\phi(v(t))|}{1+m_\eps(\sfx(t))|\nabla\phi(v(t))|}|v'(t)|+
  \frac{\langle \nabla\phi(v(t)),v'(t)\rangle }{1+m_\eps(\sfx(t))|\nabla\phi(v(t))|}=0.
\end{align*}
%
So $\delta_\eps: K_1\rightarrow \R$ can be extended to a continuously differentiable function $\delta_\eps: \H \rightarrow \R$ with gradient $\nabla\delta_\eps = Q_\eps$ on $K_1$. 
Moreover, there exists a constant $C$ only depending on $K_1$ such
that [see \cite{Hoermander98}, (2.3.8) in Theorem 2.3.6]
\begin{equation}\label{eq: 2.3.8}
\sup_{\H}|\delta_\eps| + \sup_{\H}|\nabla\delta_\eps| \leq
C\Big(\sup_{x, y\in K_1}W_\eps(x,y) + 
\sup_{x, y\in K_1}|Q_\eps(x) - Q_\eps(y)| + \sup_{K_1}|\delta_\eps| + \sup_{K_1}|Q_\eps|\Big),
\end{equation} 
where
\begin{equation*}
W_\eps(x, y):= \frac{|\delta_\eps(x) - \delta_\eps(y) - \langle
  Q_\eps(y), x- y \rangle|}{|x - y|} 
\text{ if } x\neq y, \quad W_\eps(x, x) = 0.
\end{equation*}
Since $E_\eps$ is determined up to an additive constant, we may assume
that $E_\eps(\sfu(T)) = E(\sfu(T))$ and thus that $\delta_\eps$ is
converging uniformly to $0$ on $K_1$ and $\delta_\eps\equiv 0$ on $K_1\setminus K$. Moreover, it is not
difficult to check that $Q_\eps$ is converging uniformly to $0$ on
$K_1$. Now, in order to show that $W_\eps$ is converging uniformly to $0$ on
$K_1\times K_1$, it is sufficient to prove that
$W_\eps(x_\eps, y_\eps) \to 0$ whenever $|x_\eps - y_\eps|\to 0, \
x_\eps\neq y_\eps, \ x_\eps, y_\eps \in K_1$. For this, we repeat
the arguments as in the proof of (\ref{eq: deltaeps is C1}) combined
with the argument at the end of the proof of Proposition \ref{prop:
  one dimensional setting}. 
The claim then follows. 

Therefore, we infer from (\ref{eq: 2.3.8}) that $\delta_\eps$ and $\nabla\delta_\eps$ are converging uniformly to $0$ on $\H$.  
We set 
\begin{displaymath}
\phi_\eps := \phi + \psi\delta_\eps,
\end{displaymath}
where $\psi$ has been introduced in \eqref{eq:72} and \eqref{eq: 72 ii}.
The functions $\phi_\eps: \H \rightarrow \R$ have all the desired properties. 
\end{proof}
\begin{theorem}\label{thm: De Giorgi conjecture}
  Let us suppose that $\H$ is a finite dimensional Euclidean space,
  $\phi\in \rmC^1( \H )$ satisfies 
  the quadratic lower bound \eqref{eq:coercive intro}
  and $u: [0, +\infty) \rightarrow \H$ 
  is a solution to \eqref{eq:1}. 
  Then $u$ is strongly approximable, according to Definition
  \ref{def:AGF}, i.e. 
  there exist 
  functions 
  $\phi_\tau: \H\rightarrow \R \ (\tau > 0)$ such that 
  $\Lip[\phi_\tau-\phi]\to0$ 
  as $\tau\downarrow 0$ and $\MM(\Phi,u(0)) = \{u\}=\GMM(\Phi,u(0))$ 
  for
  \begin{displaymath}
  \Phi(\tau, U, V) := \phi_\tau(V) + \frac{1}{2\tau}|V-U|^2.
\end{displaymath}
\end{theorem}     
\begin{proof}
  Lemma \ref{le:appsol} shows that the class of strongly approximable solutions is
  closed with respect to Lipschitz convergence of the functionals and locally
  uniform convergence of the solutions. 
  By Theorem \ref{thm:appmax}, every \minimal\ solution is
  strongly approximable; combining these results with the results from Lemma 
  \ref{lem: ueps is max gf for phieps} and
  \ref{lem: whitney extension} we obtain that the class of eventually
  \minimal\ solutions is also strongly approximable. 
  By Remark \ref{rem:emaxsolution} we conclude.  
\end{proof}

\begin{remark} \upshape
\normalcolor If $\phi\in\Lip(\H)$, then it clearly satisfies the quadratic lower bound \eqref{eq:coercive intro}. 
\end{remark}

\appendix
\section{Diffuse critical points for one-dimensional gradient flows}
\label{sec:example}

 In this section we give an example of a solution to a one dimensional
 gradient flow generated by a function whose derivative vanishes in a Cantor set.
 In particular, the example shows that the strict monotonicity of the energy along a solution curve 
 is not sufficient to guarantee its \minimality.
 \medskip

 Let us start from the continuous function 
 $f=-\phi': \R\rightarrow \R$ defined by
 $$
 f(x) :=
 \begin{cases}
   \pi \sqrt{x(1-x)}&\text{if $x\in(0, 1)$},\\
   0&\text{elsewhere}.
 \end{cases}
 $$
 One can check by direct calculation that
 the curve $u: [0, 1] \rightarrow \R$ 
 $$u(t):= \frac{1}{2} + \frac{1}{2} \sin\left(\pi\left(t-\frac{1}{2}\right)\right),
 \quad\text{satisfies $u'(t) = f(u(t)), \ u(0)=0, \ u(1)=1$.}$$
 Let $C\subset [0, 1]$ be the classical Cantor set and decompose $[0, 1]\setminus C$ into the disjoint union of countable open intervals $I_n=(a_n, b_n) \ (n\in\mathbb{N})$ with $l_n:=b_n - a_n$. We denote by $L_n: \R \rightarrow [0, 1]$ the continuous and piecewise linear map transforming $I_n$ into $(0, 1)$, which is constant outside $I_n$ (i.e. $L_n(x) = 0$ if $x\leq a_n$ and $L_n(x)=1$ if $x\geq b_n$).

 Since $\sum_n{l_n} = 1 < \infty$, we can choose $\beta_n > 0$ so that
 \begin{displaymath}
 \alpha_n:=\beta_n^{-1}l_n \to 0, \quad B:=\sum_n{\beta_n} < \infty. 
 \end{displaymath}
 We set 
 \begin{equation*}
 f_n(x):=\beta_n^{-1}l_nf(L_n(x))
 \end{equation*}
 and define $g:\R \rightarrow \R$ by 
 \begin{equation*}
 g(x):=\sum_n{f_n(x)}.
 \end{equation*}
 Note that $g$ is well-defined and continuous since $\sup_x \sum_{n=m}^{M}{f_n(x)}\leq \pi \sup_{n\geq m}\alpha_n$ and every $f_n$ is continuous. 

 Now, let us define the map $R:[0,1]\rightarrow \R$, 
 \begin{displaymath}
 R(x):=\sum_n{\beta_n L_n(x)}, \quad R'(x) = \frac{\beta_n}{l_n} \text{ on } I_n,
 \end{displaymath}
 which is absolutely continuous since $\sum_n{\beta_n} < \infty$, with $R'(x) > 0$ a.e.. Hence, $R$ possesses an inverse map $R^{-1}=:S: [0, B] \rightarrow [0, 1]$, which is also absolutely continuous. 

 The intervals $\tilde{I}_n := R(I_n) = (\tilde{a}_n, \tilde{b}_n)$ are disjoint, covering $[0, B]\setminus R(C)$. Note that $R(C)$ has Lebesgue measure $0$. Setting $\tilde{L}_n := L_n \circ S$, we have
 \begin{displaymath}
 \tilde{L}_n(t)=\frac{t-\tilde{a}_n}{\beta_n} \text{ if } t\in\tilde{I}_n, \ \tilde{L}_n(t) = 0 \text{ if } t \leq \tilde{a}_n, \ \tilde{L}_n(t) = 1 \text{ if } t\geq \tilde{b}_n.
 \end{displaymath}

 We define $v: [0, B] \rightarrow \R$, 
 \begin{displaymath}
 v(t):=\sum_n{l_n u(\tilde{L}_n(t))}. 
 \end{displaymath}
 It is not difficult to check that $v$ is of class $C^1$, and that 
 \begin{displaymath}
 \{t\in[0, B]: \ v'(t) = 0\} \ = \ R(C). 
 \end{displaymath}
 Moreover, it holds that $v' = g\circ v$: by density and continuity, it is sufficient to select $t\in\tilde{I}_n$; in this case, we have 
 \begin{equation*}
 v(t) = a_n + l_n u((t-\tilde{a}_n)/\beta_n)
 \end{equation*} 
 and 
 \begin{equation*}
 v'(t) = \beta_n^{-1} l_n u'(\tilde{L}_n(t)) = f_n(l_n u(\tilde{L}_n(t)) + a_n) = f_n(v(t)) = g(v(t)). 
 \end{equation*}
 So if we denote by $G$ the primitive of $g$, then $v$ is a \minimal\ gradient flow for $-G$.

 Let $\mu$ be a positive finite Cantor measure concentrated on $R(C)$, in particular $\mu(\{x\}) = 0$ for all $x$ and $\mu(R(C)) > 0$. We define 
 \begin{equation*}
 \psi(t) := t + \mu([0, t)). 
 \end{equation*}
 The map $\psi: [0, B] \rightarrow [0, B + \mu(R(C))]$ is continuous and strictly increasing and we denote by $\eta: [0, B + \mu(R(C))] \rightarrow [0, B]$ its strictly increasing inverse. For $s < t$, we have
 \begin{equation*}
 t - s = \psi(\eta(t)) - \psi(\eta(s)) = \eta(t) - \eta(s) + \mu((\eta(s), \eta(t))) \geq \eta(t) - \eta(s),
 \end{equation*}  
 i.e. $\eta$ is 1-Lipschitz continuous. 

 We define $w: [0, B + \mu(R(C))] \rightarrow \R$, 
 \begin{displaymath}
 w(s) := v(\eta(s)). 
 \end{displaymath} 
 The curve $w$ is Lipschitz continuous and $\eta'(s) = 1$ for all $s\in \psi([0, B]\setminus R(C))$. Moreover, it holds that 
 \begin{displaymath}
 \{s\in [0, B + \mu(R(C))]: \ g(w(s)) = 0\} \ = \ \{\psi(t): \ t\in[0, B], \ g(v(t)) = 0\} \ = \ \psi(R(C)).
 \end{displaymath} 
 From this we can infer 
 \begin{equation*}
 w'(s) = g(w(s)) \text{ for all } s\in(0, B + \mu(R(C))) 
 \end{equation*}
 (in particular, $w$ is of class $C^1$).
 
 The set $\psi(R(C))$ has Lebesgue measure $\mu(R(C)) > 0$. So, the gradient flow $w$ is not \minimal\ but along the curve the energy $-G\circ w: [0, B+\mu(R(C))]\rightarrow \R$ is strictly decreasing.\\

 The example could be set in a more general way, starting from a cantor-like set and an ordinary differential equation with non-uniqueness at the end points of a reference interval.

\bibliographystyle{siam}
\bibliography{bibliografia2015}

\def\cprime{$'$}
\begin{thebibliography}{10}

\bibitem{AlmgrenTaylorWang93}
{\sc F.~Almgren, J.~E. Taylor, and L.~Wang}, {\em {Curvature-Driven Flows: A
  Variational Approach}}, SIAM Journal on Control and Optimization, 31 (1993),
  pp.~387--437.

\bibitem{Ambrosio95}
{\sc L.~Ambrosio}, {\em Minimizing movements}, Rend. Accad. Naz. Sci. XL Mem.
  Mat. Appl. (5), 19 (1995), pp.~191--246.

\bibitem{AmbrosioFuscoPallara00}
{\sc L.~Ambrosio, N.~Fusco, and D.~Pallara}, {\em Functions of Bounded
  Variation and Free Discontinuity Problems}, Oxford Mathematical Monographs,
  Oxford University Press, first~ed., 2000.

\bibitem{AGS08}
{\sc L.~Ambrosio, N.~Gigli, and G.~Savar{\'e}}, {\em Gradient flows in metric
  spaces and in the space of probability measures}, Lectures in Mathematics ETH
  Z\"urich, Birkh\"auser Verlag, Basel, second~ed., 2008.

\bibitem{ball2000continuity}
{\sc J.~Ball}, {\em Continuity properties and global attractors of generalized
  semiflows and the {N}avier-{S}tokes equations}, in Mechanics: from theory to
  computation, Springer, 2000, pp.~447--474.

\bibitem{braides2012local}
{\sc A.~Braides}, {\em Local Minimization, Variational Evolution and
  $\Gamma$-Convergence}, Springer, Lecture Notes in Mathematics 2094, 2012.

\bibitem{Brezis70}
{\sc H.~Brezis}, {\em On some degenerate nonlinear parabolic equations}, in
  Nonlinear Functional Analysis (Proc. Sympos. Pure Math., Vol. XVIII, Part 1,
  Chicago, Ill., 1968), Amer. Math. Soc., Providence, R.I., 1970, pp.~28--38.

\bibitem{browder1964non}
{\sc F.~E. Browder}, {\em Non-linear equations of evolution}, Annals of
  Mathematics,  (1964), pp.~485--523.

\bibitem{DeGiorgi93}
{\sc E.~{De Giorgi}}, {\em New problems on minimizing movements}, in Boundary
  Value Problems for PDE and Applications, C.~Baiocchi and J.~L. Lions, eds.,
  Masson, 1993, pp.~81--98.

\bibitem{fleissner2016gamma}
{\sc F.~Flei{\ss}ner}, {\em Gamma-convergence and relaxations for gradient
  flows in metric spaces: a minimizing movement approach}, ESAIM Control Optim.
  Calc. Var.(to appear), arXiv preprint arXiv:1603.02822,  (2016).

\bibitem{fleissner-in-preparation}
\leavevmode\vrule height 2pt depth -1.6pt width 23pt, {\em Minimal solutions to
  generalized {$\Lambda$}-semiflows and gradient flows in metric spaces},
  submitted,  (2017).

\bibitem{GGS94}
{\sc U.~Gianazza, M.~Gobbino, and G.~Savar\'e}, {\em Evolution problems and
  minimizing movements}, Atti Acc. Naz. Lincei, 5 (1994), pp.~289--296.

\bibitem{Gobbino99}
{\sc M.~Gobbino}, {\em Minimizing movements and evolution problems in
  {E}uclidean spaces}, Annali di Matematica pura ed applicata, 176 (1999),
  p.~29–48.

\bibitem{Hoermander98}
{\sc L.~H\"ormander}, {\em The analysis of linear partial differential
  operators. {I}}, Classics in Mathematics, Springer-Verlag, Berlin, 2003.
\newblock Distribution theory and Fourier analysis, Reprint of the second
  (1990) edition [Springer, Berlin; MR1065993 (91m:35001a)].

\bibitem{leoni2009first}
{\sc G.~Leoni}, {\em A first course in Sobolev spaces}, vol.~105, American
  Mathematical Society Providence, RI, 2009.

\bibitem{MarinoSacconTosques89}
{\sc A.~Marino, C.~Saccon, and M.~Tosques}, {\em Curves of maximal slope and
  parabolic variational inequalities on nonconvex constraints}, Ann. Scuola
  Norm. Sup. Pisa Cl. Sci. (4), 16 (1989), pp.~281--330.

\bibitem{mielke2011differential}
{\sc A.~Mielke}, {\em Differential, energetic, and metric formulations for
  rate-independent processes}, Springer, 2011.

\bibitem{Mielke-Rindler09}
{\sc A.~Mielke and F.~Rindler}, {\em Reverse approximation of energetic
  solutions to rate-independent processes}, NoDEA Nonlinear Differential
  Equations Appl., 16 (2009), pp.~17--40.

\bibitem{Mielke-Rossi-Savare13}
{\sc A.~Mielke, R.~Rossi, and G.~Savar{\'e}}, {\em Nonsmooth analysis of doubly
  nonlinear evolution equations}, Calc. Var. Partial Differential Equations, 46
  (2013), pp.~253--310.

\bibitem{mielke2016balanced}
\leavevmode\vrule height 2pt depth -1.6pt width 23pt, {\em Balanced viscosity
  (bv) solutions to infinite-dimensional rate-independent systems}, Journal of
  the European Mathematical Society, 18 (2016), pp.~2107--2165.

\bibitem{Rossi-Savare06}
{\sc R.~Rossi and G.~Savar{\'e}}, {\em Gradient flows of non convex functionals
  in {H}ilbert spaces and applications}, ESAIM Control Optim. Calc. Var., 12
  (2006), pp.~564--614 (electronic).

\bibitem{Rudin91}
{\sc W.~Rudin}, {\em Functional analysis}, International Series in Pure and
  Applied Mathematics, McGraw-Hill, Inc., New York, second~ed., 1991.

\end{thebibliography}
\end{document}